\documentclass[11pt]{amsart}
\usepackage{amscd,amsfonts,amsmath,amssymb,amstext,amsthm}
\usepackage[all]{xy}

\date{}

\def\Aut{\operatorname{Aut}}

\def\dim{\operatorname{dim}}

\def\Im{\operatorname{Im}}

\theoremstyle{plain}
\newtheorem{theorem} {Theorem} [section]
\newtheorem{lemma} [theorem]{Lemma}
\newtheorem{proposition}[theorem]{Proposition}

 \theoremstyle{definition}

\newtheorem{example}  [theorem]  {Example}
\newtheorem{remark}{Remark}  [section]

\title[Schur rigidity of Schubert varieties]{Schur rigidity of Schubert varieties in rational homogeneous manifolds of Picard number one}

 \author[J. Hong]{Jaehyun Hong}
 \address{Center for Mathematical Challenges, Korea Institute for  Advanced  Study, 85 Hoegiro, Dongdaemun-gu, Seoul 02455, Korea}
 \email{jhhong00@kias.re.kr}

 \author[N. Mok]{Ngaiming  Mok}
 \address{Department of Mathematics,
The University of Hong Kong,
Pokfulam Road,
Hong Kong}
\email{nmok@hku.hk}

 \subjclass[2010]{Primary 53C30, 32G10; Secondary   14M15}
 \keywords{Schubert varieties, Schur rigidity, $\mathbb C^*$-actions, transverality}

\begin{document}

\begin{abstract}
Given  a rational homogeneous manifold $S=G/P$  of Picard number one and a Schubert variety $S_0 $ of $S$, the pair $(S,S_0)$ is said to be homologically
rigid if any subvariety   of $S$ having the same homology class  as $S_0$ must be
a translate of $S_0$ by the automorphism group of $S$.  The pair $(S,S_0)$ is said to be Schur
rigid if any subvariety  of $ S$ with homology class equal to a multiple of the homology class of $ S_0 $
must be a sum of  translates of $S_0$. Earlier we completely
determined homologically rigid pairs $(S,S_0)$ in case $S_0 $ is homogeneous
and answered the same question in   smooth non-homogeneous cases.
In this article we consider Schur rigidity, proving that
$(S,S_0)$ exhibits Schur rigidity whenever $S_0  $ is a non-linear smooth Schubert variety.

Modulo a classification result of the first author's, our proof
proceeds by a reduction to homological rigidity
by deforming a subvariety $Z$ of $  S$ with homology class equal to a multiple of the homology class of $  S_0 $ into
a sum  of   distinct  translates of $ S_0$,
 and by observing that the arguments for the homological rigidity
apply since any two translates of $S_0$ intersect in codimension at least two.
Such a degeneration   is achieved by means of   the $\mathbb C^*$-action associated with   the stabilizer of the Schubert variety  $T_0$ opposite to $S_0$. By transversality of general translates, a general translate of $Z$ intersects $T_0$ transversely and the $\mathbb C^*$-action associated with the stabilizer of $T_0$ induces a degeneration of $Z$ into a sum of translates of $S_0$, not necessarily distinct.  After investigating the Bialynicki-Birular decomposition associated with the $\mathbb C^*$-action we prove a refined form of transversality  to get a degeneration of $Z$ into a sum of distinct translates of $S_0$.
\end{abstract}

 \maketitle

\section{Introduction}

A rational homogeneous manifold $S$ is a complex projective algebraic variety on which a complex linear algebraic group $G$ acts transitively. In the current article we focus on rational homogeneous manifolds $S = G/P$ of Picard number 1. It may happen that the same projective manifold $S$ can be presented as a rational homogeneous manifold in two different ways.  In other words, taking $G$ to be identity component of the automorphism group of $S$, and $P \subset G$ to be a maximal parabolic subgroup, there may exist a simple complex Lie group $G' \subsetneq G$ such that writing $P' := P \cap G'$ the canonical map $G'/P' \hookrightarrow G/P = S$ is a biholomorphism.  Specifically, when S is a rational homogeneous manifold of type $(B_\ell,\alpha_\ell)$ (respectively, of type $(C_\ell; \alpha_1)$ or of type $(G_2,\alpha_1)$), the automorphism group of $S$ is of type $D_{\ell+1}$ (respectively, of type $A_{2\ell}$ or of type $B_3$) and thus we will regard $S$ as a rational homogeneous manifold of type $(D_{\ell+1},\alpha_{\ell+1})$ (respectively, of type $(A_{2\ell},\alpha_1)$ or of type $(B_3; \alpha_1)$).

We consider a rigidity problem on a pair $(S,S_0)$ consisting of a rational homogeneous manifold $S$ of Picard number 1 and a smooth Schubert variety $S_0 \subset S$. In the event that $S = G/P = G'/P'$ as in the above where $G' \subsetneq G$, $P' \subsetneq P$, taking a Borel subgroup $B' \subset P'$ and a Borel subgroup $B \subset P$ containing $B'$, homology groups of $S$ are generated by the finitely many Schubert varieties which are topological closures of the $B$-orbits (which are affine cells) on $S$.  {\it A priori} each $B$-orbit decomposes into the union of a finite number of $B'$-orbits.  However, since the sum of the ranks of homology groups of $S$ equals the number of $B$-orbits resp. the number of $B'$-orbits, each $B$-orbit must already be a $B'$-orbit, so that the set of Schubert varieties on $S$ does not change when $S = G'/P'$  is rewritten as $S = G/P$, which we will do throughout the article.

The ray generated by the homology class of a Schubert variety $S_0$ is extremal in the sense that if the sum   of the homology classes of two effective cycles is contained in the  ray, then both   classes are contained in the ray. It is an interesting problem in algebraic geometry to describe  the space of all effective cycles representing homology classes in this ray.
Trivially, any sum of   translates of $S_0$ by $G$ belongs to this space, but it contains more,
for example,   when $S$ is the projective linear space $\mathbb P^n$ and $S_0$ is the projective line $\mathbb P^1$ in $\mathbb P^n$.
On the other hand, when $S$ is the Grassmannian of $d$-dimensional subspaces of $\mathbb C^n$ and $S_0$ is a sub-Grassmannian, if both are non-linear, then  it consists only of sums of translates of $S_0$ (\cite{B}, \cite{Ho07}).

Given a Schubert variety $S_0$ of a rational homogeneous manifold $S$,
the pair $(S,S_0)$ is said to be {\it Schur
rigid} if any subvariety $Z \subset S$ having homology class   equal to a multiple of the homology class   of $S_0$,
 must be a sum of  translates of $S_0$ by $G$.
In this paper we consider the question of Schur rigidity.  Our main result is

 \begin{theorem}  \label{main theorem I} Let $S=G/P$ be a rational homogeneous manifold of Picard number one and let $S_0$ be a non-linear  smooth Schubert variety of $S$. Then, the pair $(S,S_0)$ is Schur rigid.
 \end{theorem}

\noindent
Here, after embedding $S$ into a projective space $\mathbb P^N$ by the ample generator of the Picard group of $S$, we say $S_0$ is linear if $S_0$ is a linear subspace of $\mathbb P^N$. Clearly, $(S,S_0)$ is not Schur rigid if $S_0$ is linear and is not maximal, i.e., there is a linear Schubert variety $S_0'$ of $S$ containing  $S_0$ properly. For the case when $S_0$ is a maximal linear space, see Proposition \ref{maximal linear space case}. \\

One of the methods to determine Schur rigidity is to use differential systems. Given a Borel subgroup of $G$, Schubert varieties of $S=G/P$ are indexed by a certain subset $W^P$ of the Weyl group $W$ of $G$. For $w \in   W^P$, let $S(w)$ denote the Schubert variety of type $w$. Kostant  constructed a representative  of the cohomology class  Poincar\'e dual to the homology class of $ S(w) $ by a closed positive $(k,k)$-form  $\phi(w)$ (\cite{Ko63}). The Schur differential system  is defined by the space of tangent  subspaces of $S$ on which $\phi(v)$ vanish for all $v$ with $\dim S(w) = \dim S(v)$ and $S(w) \not=S(v)$, and the Schubert differential  system   is defined by the space of tangent spaces of $G$-translates of $S(w)$. Then the Schur rigidity problem is reduced to two problems: (1) the equality of  the Schur differential system   and the Schubert differential system,  and (2) the uniqueness of integral varieties of the Schubert differential system up to the action of $G$.

Bryant introduced these differential systems and investigated their integral varieties for various Schubert varieties in  compact irreducible Hermitian symmetric spaces (\cite{W}, \cite{B}).   Hong proved the Schur rigidity of   smooth Schubert varieties in compact irreducible Hermitian symmetric spaces and some singular Schubert varieties in Grassmannians by reducing the problems (1) and (2) to the vanishing of certain Lie algebra cohomology spaces (\cite{Ho07}, \cite{Ho05}). Robles-The developed this method further to characterize Schur rigid Schubert varieties in compact irreducible Hermitian symmetric spaces (\cite{RoT12}, for the flexibility see \cite{CoRo13} and \cite{Ro13}).
 It is not easy to apply this differential geometric  method to Schubert varieties of  rational homogeneous manifolds other than compact Hermitian symmetric spaces:  neither of the two steps (1) and (2) can be reduced to   Lie algebra cohomology computations.

There is another form of rigidity weaker than Schur rigidity.
We say that the pair $(S,S_0)$ is {\it homologically rigid} if any subvariety $Z \subset S$ having homology class equal to $[S_0]$ must be
a $G$-translate of $S_0$. Coskun determined homologically rigid Schubert varieties of   Grassmannians and orthogonal Grassmannians by using algebro-geometric methods (\cite{Co11}, \cite{Co14}, for a survey see \cite{Co18}).

In the previous work (\cite{HoM}) we used a different method, the
geometric theory of uniruled projective manifolds based on varieties of minimal rational tangents, which was developed by Hwang-Mok and was applied to characterize uniruled projective manifolds of Picard number one and to prove their deformation  rigidity  (\cite{HwM98}, \cite{HwM99}, \cite{HwM01}, \cite{HwM02}, \cite{HwM04b}, \cite{HwM05}).
 We generalized the theory to the pair consisting of a uniruled projective manifold and one of its projective submanifold (\cite{HoM09}) and  proved that  smooth Schubert varieties in a rational homogeneous manifold of Picard number one are homologically rigid with certain obvious exceptions (Theorem 1.1 of \cite{HoM}, Theorem 1.4 of \cite{HoK}).

Recently, Mok-Zhang proved that  Schubert varieties of rational homogeneous manifolds $S$ of Picard number one, which are associated to subdiagrams of the Dynkin diagram of $S$, are Schubert rigid, i.e.,    integral varieties of their Schubert differential systems are unique  up to the action of $G$   (\cite{MkZh}).
The equality of the Schur differential system and the  Schubert differential system will imply Schur rigidity of these Schubert varieties. However, at the  moment  we don't have any tool to prove the equality of the two differential systems.

 These rigidity problems go back to the smoothability problem (\cite{B}). A homology class ${\bf s}$ in $H(S, \mathbb Z)$ is said to be {\it smoothly representable} if there is a nonsingular subvariety of $S$ whose homology class is ${\bf s}$. If the pair $(S, S_0)$ is Schur rigid and $S_0$ is singular, then any multiple $r[S_0]$ of the homology class of $S_0$ is not smoothly representable. For other kinds of smoothability and related  results, see \cite{HRT}, \cite{B}, \cite{Co11}, \cite{Co14}, \cite{Co18}.
\\

The pair $(S,S_0)$ is Schur rigid if (and only if) it is homologically rigid and
  there is no irreducible reduced subvariety of $S$ whose homology class is    $r$ times  the homology class of $S_0$ for any $r \geq 2$.
Our strategy is to deform an irreducible reduced subvariety $Z$ of $S$ representing $r[S_0]$ to a sum $Z_{\infty}$ of $G$-translates of $S_0$ such that at least one of the irreducible components has multiplicity one. We impose the condition on the multiplicity to have a neighborhood of a minimal rational curve in the smooth locus of $Z_{\infty}$.

To get such a deformation, we use the Schubert variety $T_0$ opposite to $S_0$. By transversality of general translates, there is a translate of $Z$ which intersects $T_0$ transversely. The stabilizer  of $T_0$ is a parabolic subgroup of $G$. Applying the $\mathbb C^*$-action $\lambda$ associated with this parabolic subgroup and taking its limit, we get a degeneration of $Z$ to a sum of translates of $S_0$.  But this is not enough, because irreducible components of the limit may have multiplicities $>1$. To overcome this difficulty we introduce a   transversality stronger than the usual transversality.

We   say that $Z$ intersects $T_0$ transversely with respect to the $\mathbb C^*$-action $\lambda$ if $Z$ intersects $T_0$ transversely and the $\lambda$-limits of points in the intersection $Z \cap T_0$ are all distinct.
Then the $\mathbb C^*$-action $\lambda$ gives a degeneration of $Z$ into
a sum $Z_{\infty}$  of $r$ distinct  translates of $S_0$.
 When $S_0$ is a smooth Schubert variety of a rational homogeneous manifold $S$ of Picard number one, any two translates of $S_0$ intersects in codimension $\ge 2$.  From this it follows the  existence of a line   lying on the smooth part of $Z_{\infty}$.

The problem of perturbing $Z  $
to a subvariety $Z'$  intersecting  $T_0$ transversely with respect to the $\mathbb C^*$-action $\lambda$,  is investigated in Section 3. The stabilizer $P_I^-$ of $T_0$ has an open orbit $\mathcal O$ in $T_0$, which  is an $L_I$-homogeneous bundle with fiber   a $U_I^-$-orbit $F$, where $L_I$ is the reductive part of $P_I^-$ and $U_I^-$ is the unipotent part of $P_I^-$. Then the  closure $T_0=\overline{\mathcal O}$ is  the union of $L_I$-translates of the closure $\overline{F}$. The problem is reduced to the question whether, for two points $p_1, p_2$  in $Z \cap T_0$ lying in  $F$, the isotropy of $G$ at $p_1$ can move $p_2$ in a direction outside the sum   of   the tangent space $T_{p_2} Z$ of $Z$ and the tangent space $ T_{p_2}F$ of $F$. Here, the $\mathbb C^*$-action $\lambda$ plays a role again. Especially,  the Bialynicki-Birular ($\pm$)-decomposition $\coprod_{\sigma}\mathcal O^{\pm}_{\sigma}$ associated with $\lambda$, which  is nothing but the $P_I^{\pm}$-orbit decomposition of $S$,  and the relation  among $\mathcal O_{\sigma}^+$ and $\mathcal O^-_{\tau}$ are used.   For a detailed argument see Proposition \ref{existence of perturbation}.

As we explained in the above, the relation among Schubert varieties, parabolic subgroups and $\mathbb C^*$-actions is one of the main ingredients in the proof. We collect notations, definitions, basic properties in Section 2. In Section 4 we complete the proof by generalizing the arguments for $r=1$ when the special fiber is a smooth Schubert variety to the case for $r \geq 2$ when the special fiber is   a sum of its translates (Proposition \ref{characterization-modification}) and by showing that any two translates of a smooth Schubert variety in a rational homogeneous manifold of Picard number one intersect  in codimension $\geq 2$ (Proposition \ref{I II codimension}). At the end of the paper, we prove that a maximal linear Schubert variety of $S$ is Schur rigid with the exception of  some obvious cases (Proposition \ref{maximal linear space case}).\\

This paper deals with smooth Schubert varieties in rational homogeneous manifolds of Picard number one.  While most of the arguments work in Section 3 for Schubert varieties in rational homogeneous manifolds without the conditions on smoothness and on the Picard number, these two conditions are essential in Section 4 (Proposition \ref{characterization-modification}).  While it is perceivable that generalizations to the case of smooth Schubert varieties on a rational homogeneous manifold of higher Picard number (where we have to deal with several minimal rational components) could be obtained by modifying the methods of Hong-Mok \cite{HoM} and the current article, there are intrinsic difficulties in the case of singular Schubert varieties even when the ambient rational homogeneous manifold $S$ is of Picard number 1.  In fact, it is crucial to use Kodaira Stability Theorem, which is an application of deformation theory of rational curves.  To do the same for a singular Schubert variety $S_0$ on $S$ of Picard number 1 we need at least to have an ample supply of minimal rational curves of $S$ lying on the smooth locus of $S_0$.  Unfortunately, the latter fails to be the case for certain singular Schubert varieties (cf. Remark \ref{Rem:characterization-modification}).

\section{Preliminaries}

 We  fix notations and explain properties which will be used later in this paper.
 A basic reference is \cite{Sp}.

\subsection{Parabolic subgroups}

Let $G$ be a connected simple algebraic group over $\mathbb C$.
  Fix a Borel subgroup $B$ of $G$ and a maximal torus $T$  in $B$.   Denote by $\Delta^+$ the system of positive roots of $G$ and by $\Phi=\{ \alpha_1, \dots, \alpha_{\ell}\}$ the system of simple roots of $G$.
  For any subgroup $H$ of $G$ invariant under the conjugation by $T$, let $\Delta(H)$ denote the set of all roots $\alpha$ whose root space $\frak g_{\alpha}$ is contained in the Lie algebra of $H$. For example, $\Delta(B) = \Delta^+$.

For a subset $I$ of $\Phi$ denote by $P_I^{\pm}$ the   parabolic subgroup of $G$ whose Lie algebra is  $$(\frak t + \sum_{ \alpha \in \mathbb ZI \cap \Delta } \frak g_{\alpha}) + \sum_{\alpha \in \Delta^+ \backslash \mathbb Z I} \frak g_{\pm\alpha}.$$  Let  $L_I$ be its reductive part containing $T$ and $U_I^{\pm}$ be its  unipotent  part. Then $\Delta(L_I)=\Delta \cap \mathbb Z I$ and $\Delta(U_I^{\pm}) =\Delta^{\pm} \backslash \mathbb Z I$.  For $I=\emptyset \subset \Phi$, $P_{\emptyset}^{\pm}$ is a Borel subgroup $B^{\pm}$ of $G$ ($B^+$ is the given Borel subgroup $B$ and $B^-$ is the Borel subgroup opposite to $B$) and $U_{\emptyset}^{\pm}$ is the unipotent part $U^{\pm}$ of $B^{\pm}$.
As in the case of $B$, we will sometimes use   notations $P_I$, $U_I$, $U$ instead of  $P_I^+$, $U_I^+$, $U_{\emptyset}^+$.

    To each simple root $\alpha_k$ we associate a parabolic subgroup $ P_{\Phi -\{\alpha_k\}}$ of $G$.
 The homogeneous manifold  $S=G/P_{\Phi-\{\alpha_k\}}$ is called the rational homogeneous manifold associated to the simple root $\alpha_k$ or of type $(G, \alpha_k)$. From now on we fix a simple root  $\alpha_k$ and set $P=P_{\Phi -\{\alpha_k\}}$,  $L_P=L_{\Phi -\{\alpha_k\}}$,  $U_P=U_{\Phi -\{\alpha_k\}}$ and $S=G/P$.

\subsection{Schubert varieties}

 Let $  W$ be the Weyl group of $G$ and let $W_P$ be the Weyl group of the reductive part of $P$. Then  the set of $T$-fixed points in $S=G/P$ is indexed by the  set of right cosets $W/W_P$: letting $x_0$ be the base point of $S $ with the isotropy group $P$,  the map $[w] \in W/W_P \mapsto x_w:=wx_0$ is a bijective map from $W/W_P$ to the $T$-fixed point set in $S$. The $B$-orbit decomposition of $S$ is given by $S=\coprod _{[w] \in   W/W_P} B.x_w$.

 For an element $w \in  W$ define a subset $\Delta(w)$ of $\Delta^+$ by $\Delta(w)=\{ \beta \in \Delta^+  :  w(\beta) \in -\Delta^+\}$.
Let $ W^P$ be the subset of  $  W$  consisting of
 $  w \in   W$ such that $ \Delta(w) \subset \Delta(U_P) $.   Then $W^P$ is a set of representatives of   $W/W_P$ so that we have
  $$S=\coprod _{w \in   W^P} B.x_w  .$$
  For each $w \in   W^P$, the closure $S(w)$  of $B. x_w$ is called  the {\it Schubert variety of type} $w$. For the Borel subgroup $B^-$ opposite to $B$ we also have a   decomposition $S=\coprod _{w \in   W^P} B^-.x_w$, and we call the closure $T(w)$ of $B^-.x_w$   the {\it opposite Schubert variety  of type} $w  $, or, the {\it Schubert variety opposite to} $S(w)$. \\

Each orbit  $B^{\pm}.x_w$ is biholomorphic to an affine cell $\mathbb C^{k^{\pm}}$ for some $k^{\pm}$ depending on $w$, so that we get a coordinate system on it.
We describe this coordinate system more precisely. Consider a root $\alpha$ as a  character of $T$, i.e., a homomorphism from $ T $ to the multiplicative group $\mathbb C^* $, so that, for a root vector $e_{\alpha} \in \frak g_{\alpha}$, we have $Ad(h)e_{\alpha} =\alpha(h)e_{\alpha}$ for $h \in T$. Then   there exist  a unique closed subgroup $U_{\alpha} \subset G$ and an isomorphism $u_{\alpha}: \mathbb C \rightarrow U_{\alpha}$ such that $\Im du_{\alpha} =\frak g_{\alpha}$ and  for $h \in T$
$$hu_{\alpha}(x)h^{-1} =u_{\alpha}(\alpha(h)x)$$
 where $x \in \mathbb C $ (8.1.1 of \cite{Sp}).

Let $U_P^-$ be the unipotent part of $P^-$, the parabolic subgroup of $G$ opposite to $P$. Then $ U_P^-.x_0$ is an open neighborhood of $x_0$ biholomorphic to $U_P^-=\prod_{ \alpha \in \Delta(U_P )} U_{-\alpha}$.
Thus, for $w \in   W^P$, $w(U_P^-).x_w=wU_P^-.x_0$ is an open neighborhood of $x_w$ in $S=G/P$ biholomorphic to $w(U_P^-)=\prod_{ \alpha \in \Delta(U_P )} U_{-w(\alpha)}$.
Similarly,  $B.x_w=(U \cap w(U_P^-)).x_w$ is an open neighborhood of $x_w$ in $S(w)$ biholomorphic to $U \cap w(U_P^-)=w\left(\prod_{\beta \in \Delta(w)}U_{-\beta}\right)=\prod_{\alpha \in \Delta(w^{-1})} U_{\alpha}$
 and $B^-.x_w=(U^- \cap w(U_P^-)).x_w$ is an open neighborhood of $x_w$ in $T(w)$ biholomorphic to $U^- \cap w(U_P^-)$.

Let $A$ be a set consisting of roots. By taking the product of the isomorphisms $u_{\alpha}$ where $\alpha \in A$, we get  an isomorphism
$$u=\prod_{\alpha \in A} u_{\alpha}:\mathbb C^{|A|} \rightarrow \prod_{\alpha \in A} U_{\alpha} $$
and $u$ defines coordinates $(x_{\alpha})_{\alpha \in A}$ on $\prod_{\alpha \in A} U_{\alpha} $. We remark that the coordinates depend on the order of roots in $A$.
By taking an appropriate set $A$ of roots, we get coordinates systems on $w(U^-_P).x_w=wB.x_0$, $B.x_w$, and $B^-.x_w$.\\

For a subset  $I $  of  $\Phi$ let $W_I$ be the Weyl group of $L_I$. Then the left coset space  $W_I \backslash W^P$   parameterizes $P_I^{\pm}$-orbits in $S=G/P$ so that we have
$$S=\coprod_{[\sigma]\in   W_I \backslash    W^P} P_I^{\pm}.x_{\sigma}$$
(8.4.6 of \cite{Sp}).
The closure of a $P_I$-orbit in $S$ is a Schubert variety and the closure of a $P_I^-$-orbit in $S$ is an opposite Schubert variety (For,  an irreducible $B$-invariant   subvariety of $S$ is the closure of a $B$-orbit and an irreducible $B^-$-invariant   subvariety of $S$ is the closure of a $B^-$-orbit).

 Conversely,  the stabilizer of a Schubert variety in $G$ is a parabolic subgroup $P_I$ for some $I \subset \Phi$.

 \begin{proposition} \label{the base has positive dimension}
 For  $w \in W^P$, let $I=\Phi \cap w(\Delta(  P^-))$. Then
 \begin{enumerate}
 \item the stabilizer of $S(w)$ in $G$ is $P_I$;
 \item $L_I.x_w$ is closed, and  has positive dimension if $w \not=Id$.

 \end{enumerate}

 \end{proposition}
 \begin{proof}

 (1) (Section 2.6 of \cite{BrPo99}).    For a simple root $\alpha$,  the minimal parabolic subgroup $P_{\alpha}$ acts invariantly on $BwB$ if and only if
$w^{-1}(\alpha)<0$, or equivalently, $\ell(s_{\alpha}w) = \ell(w) -1$. From $s_{\alpha}w=ws_{w^{-1}(\alpha)}$ it follows that $P_{\alpha}$ acts invariantly on $S(w)$ for any simple root $\alpha$ with $w^{-1}(\alpha) \in \Delta(L_P) $. Thus  the stabilizer group of $S(w)$ is $P_{ I } $  where $ I =\{ \alpha_1, \cdots, \alpha_{\ell} \} \cap  w  (\Delta(P^-)) $.

(2) $L _I\cap B$ is contained in $L_I \cap wP^-w^{-1}$ and thus  $L _I\cap B^-$ is contained in  the isotropy $L_I \cap wPw^{-1}$  of $L_I$ at $x_w=wx_0$. Therefore, the $L_I$-orbit $L_I.x_w$ is closed.

  If $ w  \not = id$, there is a simple root $\alpha$ such that $ w ^{-1}(\alpha)<0$. Then $\alpha$ is contained in $I =\{ \alpha_1, \cdots, \alpha_{\ell} \} \cap  w  (\Delta(P^-))$.
Then,  $w^{-1}(\alpha) $ is contained in $\Delta(U_P^-)$
 because any $\beta <0$ with $w(\beta) >0$ is contained in $\Delta(U_P^-)$ by definition of $ W^P$.
Thus $U_{\alpha}  \subset (L_I \cap U ) \cap  w U_P^-  w ^{-1}$ and
 $ x _{ w } \not= U_{ \alpha} . x _{ w } \subset L_I. x  _{ w }$.
   \end{proof}

\subsection{$\mathbb C^*$-actions}

Let $\lambda: \mathbb C^* \rightarrow T$ be  a cocharacter of $T$,  i.e., a homomorphism from the multiplicative group $\mathbb C^*$ to $T$. Then $\lambda$ induces a $\mathbb C^*$-action on $G$: $t.g = \lambda(t)g\lambda(t)^{-1}$ for $t \in \mathbb C^*$ and $g \in G$,
which again induces a $\mathbb C^*$-action on $G/P$ because $T$ normalizes $P$.
Then  for a root $\alpha$ and the isomorphism $u_{\alpha}: \mathbb C \rightarrow U_{\alpha}$,  we have
$$t.u_{\alpha}(x)=\lambda(t)u_{\alpha}(x) \lambda(t)^{-1} = u_{\alpha}(t^{\langle \alpha, \lambda \rangle}x) $$
where  $t \in \mathbb C^*$ and $x \in \mathbb C$.

Let $A$ be a set consisting of roots.
In   coordinates $(x_{\alpha})_{\alpha \in A}$ on $\prod_{\alpha \in A} U_{\alpha} $, the $\mathbb C^*$-action induced by $\lambda$ on $\prod_{\alpha \in A} U_{\alpha}$ can be written as
$$t.(x_{\alpha})_{\alpha \in A} = (t^{\langle \alpha, \lambda \rangle}x_{\alpha})_{\alpha \in A}$$ for $t \in \mathbb C^*$ and $(x_{\alpha})_{\alpha \in A} \in \mathbb C^{|A|}$.\\

Let $P(\lambda)$ be the subgroup  of $G$ consisting of $g \in G$ such that $\lim_{t \rightarrow 0}t.g$ exists. Then, $P(\lambda)=P_I$ for some choice of a Borel subgroup of $G$ where  $I$ is the set of simple roots orthogonal to $\lambda$, and thus is a parabolic subgroup of $G$ (8.4.5 of \cite{Sp}). Conversely, any parabolic subgroup of $G$ is of the form $P(\lambda)$ for some cocharacter $\lambda$. Define $P(-\lambda)$ by the subgroup of $G$ consisting of $g \in G$ such that $\lim_{t \rightarrow \infty} t.g$  exists.   If $P(\lambda) =P_I$, then $P(-\lambda) = P_I^-$.

One   way  to understand the $P_I^{\pm}$-orbit decomposition on $S=G/P$   is as the ($\pm$)-decomposition for the associated $\mathbb C^*$-action $\lambda$ on $S$. Before  giving this relation, we state   Bialynicki-Birula decomposition theorem on an arbitrary projective manifold.
 For a fixed point $x  $ of a $\mathbb C^*$-action on a projective manifold $X$, noting that $T_xX$ is a $\mathbb C^*$-module, let $(T_xX)^{\pm}$ denote the $\mathbb C^*$-submodule of $T_xX$ spanned by    all vectors $v \in T_xX$ such that for $t \in \mathbb C^*$ we have $t.v=t^mv$ for some $m \in \mathbb Z_{\pm}$.

\begin{proposition} [Theorem 4.3  of \cite{BB}]  \label{BB decomposition}
Given a $\mathbb C^*$-action   on a nonsingular projective variety $X$ let $ \coprod_{i=1}^r  X_i$ be the decomposition of its fixed point set   into connected components.  Then   there   are   canonical decompositions $X=\coprod_{i=1}^r X_i^+$ and $X=\coprod_{i=1}^r X_i^-$ of $X$ into  locally closed nonsingular $\mathbb C^*$-invariant subvarieties and $\mathbb C^*$-fibrations $\gamma_i^+:X_i^+ \rightarrow X_i$ and $\gamma_i^-:X_i^- \rightarrow X_i$ for $i=1, \dots, r$ such that  the fixed point set of $X_i^{\pm}$ is $X_i$ and
 for $x \in X_i$, $T_{x}X_i^{\pm} =T_xX_i \oplus (T_xX)^{\pm}$
  for $i=1, \dots, r$.
Furthermore, such decompositions are unique.

\end{proposition}

The  decomposition $X=\coprod_{i=1}^r X_i^+$ (respectively,  $X=\coprod_{i=1}^r X_i^-$) is called the ($+$)-decomposition (respectively, the ($-$)-decomposition). For any $i=1, \dots, r$
\begin{eqnarray*}
X_i^+ &=&\{x \in X: \lim_{t \rightarrow 0}t.x \in X_i\}\\
X_i^- &=&\{x \in X: \lim_{t \rightarrow \infty}t.x \in X_i\}
\end{eqnarray*}
and $X_i^+$ (respectively, $X_i^-$) is the stable (respectively, unstable) subvariety of $X$ corresponding to $X_i$ (\cite{BB76}).

\begin{proposition} \label{BB decomposition for S}
Let $\lambda$ be a cocharacter of $T$ with $P(\lambda) = P_I$. The
 $(\pm)$-decomposition of $S=G/P$ under the $\mathbb C^*$-action induced by $\lambda$ is
$$S=\coprod_{[\sigma]\in   W_I \backslash    W^P} P_I^{\pm}.x_{\sigma} $$
 together with the projections $\gamma^{\pm}_{[\sigma]}:P_I^{\pm}. x_{\sigma} \rightarrow L_I.x_{\sigma}$.
Furthermore, there is a unique $[\sigma^+ ] \in   W_I \backslash   W^P$ such that $P_I^-.x_{\sigma^+}$ is closed. In this case, $P_I^+.x_{\sigma^+}$  is open in $S$ and it intersects $P_I^-.x_{\tau}$ nontrivially for any $\tau \in W^P$.
\end{proposition}

\begin{proof} The decompositions $S=\coprod_{[\sigma] \in W_I\backslash W^P} P_I^+.x_{\sigma}$ and $S=\coprod_{[\sigma] \in W_I\backslash W^P} P_I^-.x_{\sigma}$ and projections $\gamma_{[\sigma]}^+:P_I^+.x_{\sigma} \rightarrow L_I.x_{\sigma}$ and $\gamma_{[\sigma]}^-:P_I^-.x_{\sigma} \rightarrow L_I.x_{\sigma}$ satisfy the conditions in Proposition \ref{BB decomposition}.
The first statement follows from the uniqueness of the ($\pm$)-decomposition.

 If $P_I^-.x_{\sigma}$ is closed, then $U_I^-$ acts on $P_I^-.x_{\sigma}$ trivially and thus we have $P_I^-.x_{\sigma}=L_I.x_{\sigma}$   and   $\dim (T_xS)^-=0$ for $x \in L_I.x_{\sigma}$.
Therefore, $P_I^+.x_{\sigma}$ is open in $S$, and since the decomposition is locally closed, there is a unique $[\sigma^+] \in W_I\backslash W^P$ such that $P_I^+.x_{\sigma^+}$ is open (cf. Corollary 1 of \cite{BB}). Therefore, $P_I^-.x_{\sigma^+}$ is the unique closed $P_I^-$-orbit in $S$.

Suppose that $P_I^+.x_{\sigma^+} $ is disjoint from $P_I^-.x_{\tau}  $ for some $\tau \in W^P$. Then $P_I^-.x_{\tau}$ is contained in     the complement $Z$ of  $P_I^+.x_{\sigma^+}$ in $S$. Since $Z$ is $P_I^+$-invariant,  $P_I^+P_I^-.x_{\tau}$ is contained in $Z$.  However, $P_I^+P_I^-$ is open in $G$, and thus $P_I^+P_I^-.x_{\tau}$ is open in $S$, contradicting to the fact  that $Z$  is a proper subvariety of $S$. Therefore, $P_I^+.x_{\sigma^+}$ intersects $P_I^-.x_{\tau}$ for every $\tau \in W^P$.
\end{proof}

\begin{remark} For $\sigma, \tau \in W^P$,
$P_I^-.x_{\sigma} $ is in the closure of $P_I^-.x_{\tau}$ if and only if $P_I^+.x_{\sigma} \cap P_I^-.x_{\tau}$ is nonempty  (Corollary 1.2 of \cite{Deo}, Section 1.3 of \cite{Br}).
\end{remark}

\section{Transversality with respect to a $\mathbb C^*$-action}

Let $\mathcal O$ be a $P_I^-$-orbit  in $S=G/P$ and let $\overline{\mathcal O}$ be its closure in $S$. By the transversality of a general translate (\cite{K}), for any subvariety $Z$ of $S$ of complementary dimension, there is a translate $gZ$ of $Z$ by an element $g\in G$ which intersects $\overline{\mathcal O}$ transversely. Then the intersection multiplicity of $gZ$ and $\overline{\mathcal O}$   at each intersection point is one, so that we have $gZ \cap \overline{\mathcal O }=p_1 +\dots+ p_r$ where all $p_i$ are distinct.
Take the limit $\lim_{t \rightarrow \infty}t.p_i:=x_i$ for the $\mathbb C^*$-action $\lambda $ associated to $P_I^-$ (the one induced by a cocharacter $\lambda$ of $T$ such that $P(-\lambda)=P_I^-$). Then $x_i$'s are not necessarily distinct any more.
The goal of this section is to prove that there is a translate $ Z'$ of $Z$ such that
\begin{enumerate}
 \item $ Z'$ intersects $\mathcal O$ transversely and
  \item the limits $\lim_{t\rightarrow 0}t.p_i'$ are all distinct, where $p_1'+ \dots + p_r'= Z' \cap \mathcal O$
\end{enumerate}
(see Proposition \ref{existence of perturbation}).
\noindent
Let $\gamma:\mathcal O \rightarrow \mathcal H$ be the projection map to the fixed point set $\mathcal H$ in $\mathcal O$ of the $\mathbb C^*$-action $\lambda$.
Then the condition (2) is equivalent to the condition that   $\gamma_*( Z' \cap \mathcal O)=x_1 +\dots +x_r$  with all $x_i$ being distinct. \\

Because the complement of $\mathcal O$ in $\overline{\mathcal O}$ has dimension  less than $\dim \mathcal O$, we can find a translate of $Z$ satisfying the condition (1), if we apply the transversality Theorem of the following form.
For subvarieties  $Y, Z$  of   $S $, we say that $Y$ meets $Z$ {\it properly} if for each irreducible component $C$ of $Y \cap Z$, $\text{codim} (C) = \text{codim} Y + \text{codim} Z$.   
Recall that  as a linear algebraic group $G$ is an affine algebraic manifold.

\begin{proposition} [\cite{K}, Lemma 1.3.1 of \cite{Br}] \label{transversality Lemma} Let $Y, Z$ be subvarieties of a rational homogeneous manifold $S = G/P$.  Let $Y_0 \subset Y$ (resp. $Z_0 \subset Z$) be the nonsingular locus of $Y$ (resp. $Z$).  Then, there exists a nonempty Zariski open subset $W$ of $G$ such that for any group element $g$ belonging to $W$, $Y$ meets $gZ$ properly, and $Y_0 \cap gZ_0$ is nonsingular and dense in $Y \cap gZ$.  In particular, if $\dim(Y) + \dim(Z) = \dim(S)$, then $Y$ meets $gZ$ transversely for all $g$ belonging to some dense Zariski open subset of $G$.
\end{proposition}

To get a translate of $Z$ satisfying the condition (2) we use a birational morphism $f: \widetilde {\mathcal O} \rightarrow \overline{\mathcal O}$ constructed in the following way. %
Let $\mathcal H$ be the fixed point set of the $\mathbb C^*$-action $\lambda$ in $\mathcal O$. Then $\mathcal H$ is an $L_I$-orbit. Take a base point $o$ of $\mathcal H$ and put $F$  to be the $U_I^-$-orbit of $o$.
 Then $\gamma: \mathcal O \rightarrow \mathcal H$ is the $L_I$-homogeneous fiber bundle over $\mathcal H$ with typical fiber $F$. Take the closure $\overline F$ of $F$ in $S$. Then the homogeneous fiber bundle $\widetilde {\mathcal O}$ over $\mathcal H$ with fiber $\overline F$ has $\mathcal O$ as a dense open subset. Define a map
 \begin{eqnarray*} 
\xymatrix{
\widetilde{\mathcal O}=L_I \times_{P_1} \overline{F} \ar[d]^{\widetilde{\gamma}}  \ar[r]^{\qquad \quad    f } & \overline{\mathcal O}\\
\mathcal H=L_I/P_1 &  }
\end{eqnarray*}
by $f([\ell, x])=\ell x$ for $\ell \in L_I$ and $x \in \overline{F}$, where $P_1$ is the isotropy group of $L_I$ at $o$. Then $f$ is a birational morphism.
  For $z \in \mathcal H$, the restriction of $f$ to the fiber $\widetilde{\gamma}^{-1}(z)$ of $\widetilde{\gamma}$ is a closed embedding into $S$. We denote the image $f(\widetilde{\gamma}^{-1}(z))$ by $\overline{F}_z$. Then $\overline {\mathcal O}$ is the union of $  \overline{F}_z$, where $z \in \mathcal H$. Thus, for $y \in \overline{\mathcal O}$, $f^{-1}(y)$ is biholomorphic to the subvariety of $\mathcal H$ consisting of points  $z \in \mathcal H$ with $y \in \overline{F}_z$. Therefore, $f$ is proper.

\begin{remark} The proper morphism $f:\widetilde {\mathcal O} \rightarrow \overline{\mathcal O}$  is a generalization of  a   collapsing of a homogeneous vector bundle   in   \cite{Ke}.

\end{remark}

Set $\mathcal E$ to be the subset $\{y \in \overline{\mathcal O}: \dim f^{-1}(y) >0\}$ of $\overline{\mathcal O}$. Since $\overline{\mathcal O}$ is a Schubert variety, it is normal, and thus the restriction of $f$ to $\widetilde {\mathcal O} - f^{-1}(\mathcal E)$ is a biholomorphism onto $\overline{\mathcal O} -\mathcal E$.  Therefore, $f^{-1}(\mathcal E)$ is the exceptional locus of $f$ and $\mathcal E$ is the image $f(\mathrm{Exc}(f))$ of the exceptional locus of $f$.


\begin{lemma} \label{pushing out}  Let $\overline {F}$ and  $\mathcal E$ be as in the  above.
  For any $x \in \overline F \cap \mathcal E$ and $y \in F$,
there is a vector field $X$ of $TS$ such that $X(x)=0$ and $X(y) \not=0 \in  T_{y}(L_I.y)$.
\end{lemma}


\begin{proof} For $x \in \overline{F} $, let $P_x$ be the isotropy group of $L_I$ at $x$. Then $x$ is contained in $\mathcal E$ if and only if $P_x.o$ has positive dimension,  because  $\mathcal E$ consists of $x \in \overline{\mathcal O}$ such that $ \dim f^{-1}(x) >0$.  Thus for $x \in \mathcal E \cap \overline F$, there is  a one parameter subgroup $(\ell_{\epsilon})$ of $P_x$ such that $\frac{d}{d\epsilon}|_{\epsilon=0} \ell_{\epsilon}.o \not=0$. Furthermore, for any subgroup $P'$ of $L_I$ and for any $y \in F$,  $P'.o$ has positive dimension if and only if $P'.y$ has positive dimension. Therefore, for $x \in \mathcal E \cap \overline F$ and $y \in F$, there is a one parameter subgroup $(\ell_{\epsilon})$ of $P_x$ such  that $\frac{d}{d\epsilon}|_{\epsilon=0} \ell_{\epsilon}.y \not=0$.
\end{proof}

\begin{proposition} \label{existence of perturbation}  Let $S=G/P$ be a rational homogeneous manifold     and let $\mathcal O$ be a $P_I^-$-orbit in $S$.    Assume that   the exceptional locus   of $f:\widetilde {\mathcal O} \rightarrow \overline{\mathcal O}$ associated to $\mathcal O$ is not empty.
Let $Z$ be a  reduced irreducible subvariety of $S$ of dimension $\dim S -\dim \mathcal O$ with nonzero intersection number with $\overline{\mathcal O}$.  Then there exists  $g \in G$ such that
\begin{enumerate}
\item $gZ$ intersects $\mathcal O$ transversely and
\item $\gamma_*(gZ\cap \mathcal O)=  x_1 + \dots +  x_r$ with all $x_i$ being distinct,
\end{enumerate}
\noindent
    where $\gamma:\mathcal O \rightarrow \mathcal H$ is the projection map to the fixed point set $\mathcal H$
\end{proposition}

\begin{proof}  Without loss of generality, up to the action of $G$, we may assume that  $Z$ intersects $\overline{\mathcal O}$ transversely and the intersection $Z \cap \overline {\mathcal O} $ is contained in $\mathcal O$ (Proposition \ref{transversality Lemma}).
 Then $Z$ intersects $\mathcal O$ at $r$ distinct points. If $\mathcal O$ is closed, then $\mathcal O=\mathcal H$  and  there is nothing to prove. From now one we assume that $\mathcal O$ is not closed and thus $\gamma:\mathcal O \rightarrow \mathcal H$ has positive fiber dimension.
Let $F$ be the $U_I^-$-orbit of the base point $o$ of $\mathcal H$  and $\mathcal E$ be the image $f(\mathrm{Exc}(f))$ of the exceptional locus of $f$ as in Lemma \ref{pushing out}.

Write $\gamma_*(Z \cap \mathcal O) =n_1x_1 + \dots + n_sx_s$ (where all $x_i$ are distinct). Put $m_0 =0$ and $m_k =m_{k-1}+n_k$ for $k=1, \dots, s$. We order points $p_1, \dots, p_r$ in $Z \cap \mathcal O$ in such a way that $p_i$ lies over $x_k$ for  $m_{k-1}+1 \leq i \leq m_k$, $k=1, \dots, s$. Take a neighborhood $\mathcal U_k$ of $x_k$ in $\mathcal H$ for $1 \leq k \leq s$ such that $\mathcal U_k \cap \mathcal U_l = \emptyset$ for any $1 \leq k \not=l \leq s$.
For $m_{k-1} +1 \leq i \leq m_k$, take a neighborhood $\mathcal B_i$ of $p_i$ such that $\gamma(\mathcal B_i) \subset \mathcal U_k$ for any $m_{k-1}+1 \leq i \leq m_k$ and $ \mathcal B_{i}  \cap    \mathcal B_{j} =\emptyset$  for any $m_{k-1} +1 \leq i\not=j \leq m_k$.

   Let $\{g_{\epsilon}\}$ be a one parameter subgroup of automorphisms of $S$.  Then, for sufficiently small $\epsilon$,  the intersection  $g_{\epsilon}Z \cap \mathcal O$ is contained in the union $\coprod _{i=1}^r \mathcal B_i$.  We order the points $p_{1,\epsilon}, p_{2, \epsilon}, \dots, p_{r, \epsilon}$ in the intersection $g_{\epsilon}Z \cap \mathcal O$    in such a way that $p_{i, \epsilon} \in \mathcal B_i$ for all $i=1, \dots, r$. By the construction of $\mathcal B_i$, $p_{i, \epsilon}$ and $p_{j, \epsilon}$ lie in   different fibers if $p_i$ and $p_j$ lie in  different fibers. Therefore, it suffices to show that for any two points $p_i$ and $p_j$ in the same fiber, there is $g_{\epsilon}$ such that $p_{i, \epsilon}$ and $p_{j, \epsilon}$ lie in  different fibers.

If $p_i$ were in   $\mathcal E$, then, by Lemma \ref{pushing out}, there would be a one-parameter subgroup $(\ell_{\epsilon})$ of $L_I$ such that $\ell_{\epsilon}$ fixes $p_i$ and $\ell_{\epsilon}$ sends $p_j$ to a different fiber. However, our $p_i$ is not in $\mathcal E$ and we cannot apply this lemma directly. \\

If $n_k=1$ for all $k=1, \dots, s$, then there is nothing to prove. Suppose that $n_1 >1$.
We will use the same notations as in Lemma \ref{pushing out}.  Take two points $p_1, p_2 $ in $(Z \cap \overline{\mathcal O}) \cap \gamma^{-1}(x_1)$. We may assume that $p_1,p_2$ lie in $F$.
Consider the limit $p_0:=\lim_{t\rightarrow 0} t.p_1  $, where  the action of $t \in \mathbb C^*$ is given by the $\mathbb C^*$-action $\lambda$ associated to $P_I^-$.

Assume that  $p_0 $ belongs to $ \mathcal E$. Then,
by Lemma \ref{pushing out}, for each $t \in \mathbb C^*$, there is a vector field $X$ of $TS$  such that $X(p_0)=0$ and $X(t.p_{2}) \not \in T_{t.p_{2}}t.Z  + T_{t.p_{2}}F$. Thus, for sufficiently small $t$, there is a vector field $X'$ of $TS$ such that $X'(t.p_1)=0$ and $X'(t.p_{2}) \not \in T_{t.p_{2}}t.Z + T_{t.p_{2}}F$.
By taking such a $t \in \mathbb C^*$ and by replacing $Z$ by $t.Z$, we may assume that there is a vector field $Y$ of $TS$ such that $Y(p_1)=0$ and $Y(p_2) \not \in T_{p_2}Z + T_{p_2}F$.

 Let $\{g_{\epsilon}\}_{\epsilon \in \mathbb C}$ denote the one parameter subgroup of automorphisms of $S$ associated to $Y$.  Let $p_{1,\epsilon}, p_{2, \epsilon}, \dots, p_{r, \epsilon}$ be points in the intersection $g_{\epsilon}Z \cap \mathcal O$, ordered in such a way as at the beginning of the proof.
Then, $p_{1, \epsilon} =p_1 \in F$ for all $\epsilon$.   We will show that $p_{2, \epsilon} \not \in F$ for some $\epsilon$.

Suppose that $p_{2, \epsilon}  \in F$ for all $\epsilon$. Then $\frac{d}{d\epsilon}|_{\epsilon =0} p_{2, \epsilon}    \in T_{p_{2}}F$.  Write $p_{2,\epsilon}=g_{\epsilon}z_{\epsilon}$, where $z_{\epsilon} \in Z$. Since $g_{0}=$ the identity and $p_{2, 0}=p_{2}$, we have $z_{0} = p_{2}$. From
$$\frac{d}{d\epsilon}|_{\epsilon =0} p_{2, \epsilon}    = \frac{d}{d\epsilon}|_{\epsilon =0}g_{\epsilon}p_{2} + \frac{d}{d\epsilon}|_{\epsilon =0} z_{\epsilon} \in T_{p_{2}}F$$
 and  $\frac{d}{d\epsilon}|_{\epsilon =0} z_{\epsilon} \in T_{p_{2}}Z$,
  it follows that we have $Y(p_2)=\frac{d}{d\epsilon}|_{\epsilon =0}g_{\epsilon}p_{2}  \in  T_{p_{2}}Z + T_{p_{2}}F  $, which contradicts to the assumption $Y(p_{2}) \not \in T_{p_{2}}Z + T_{p_{2}}F$. Therefore, $p_{2, \epsilon} \not \in F$ for some $\epsilon$. \\

If   $\mathcal E$ is equal to the  boundary of $\mathcal O$, then any limit $p_0=\lim_{t \rightarrow 0}t.p_1$ lies in $\mathcal E$ and by the arguments in the above there is $g_{\epsilon} \in G$ such that $p_{1, \epsilon}$ and $p_{2, \epsilon}$ lie in different fibers. This is the case if, for example, $\overline{\mathcal O}$ is an odd symplectic Grassmannian embedded into a symplectic Grassmannian.
However,   $\mathcal E$ is not necessarily equal to the boundary of $\mathcal O$.

 \begin{lemma} \label{key lemma} Assume that 
 $\mathcal E$ is nonempty.
Then the set $\mathcal E^+ _F$ of points $p \in F$ such that the limit $\lim_{t \rightarrow 0}t.p $ is contained in $\mathcal E $ is nonempty.

 \end{lemma}

 \begin{proof}
 The $P^{\pm}_I$-orbit decomposition of $S=G/P$  is given by $$S=\coprod_{[\sigma] \in W_I \backslash   W^P} \mathcal O^{\pm}_{\sigma}$$
 where
 $\mathcal O^{\pm}_{\sigma}$ is $P^{\pm}_I.x_{\sigma}$. Our $P_I^-$-orbit $\mathcal O$ is   $\mathcal O^-_w$ for some $w \in   W^P$.
Let $\Sigma:=\mathcal O_{\sigma }^-$ be the unique closed  $P_I^-$-orbit in $S$ (Proposition \ref{BB decomposition for S}).
 Since $\Sigma$ is closed, it is in fact equal to $\mathcal H_{\sigma }$, the fixed point set in $\mathcal O_{\sigma}^-$.
 Thus any point in $\mathcal O^+_{\sigma }$ has limit in $\Sigma=\mathcal H_{\sigma}$ as $t$ goes to $0$.  By Proposition \ref{BB decomposition for S}, $\mathcal O^+_{\sigma }$ has nonempty intersection with $\mathcal O^-_{w}$. Since $\mathcal O^+_{\sigma} \cap \mathcal O^-_w$ is invariant under the action of $L_I$, $\mathcal O^+_{\sigma}$ has nonempty intersection with $F$, too.

 Since  $\mathcal E$   is closed and   invariant under the action of $P_I^-$, and    is nonempty,   the unique closed $P_I^-$-orbit $\Sigma$ is contained in $\mathcal E$ and hence $\mathcal E_F^+$ is nonempty.
 \end{proof}

By Lemma \ref{key lemma}, there is an element $u$ of $U_I^-$ such that $u.p_1$ belongs to $\mathcal E^+_F $.  By replacing $Z$ by $u.Z$, we may assume that $p_1$ has the property that the limit $p_0:=\lim_{t\rightarrow 0} t.p_1  $ belongs to $ \mathcal E$.
The proof goes in the same line as before, so that we get $g_{\epsilon} \in G$ such that $p_{1, \epsilon}$ and $p_{2,\epsilon}$ lie in different fibers. \\

   Now the number of intersection points in $g_{\epsilon}Z \cap \mathcal O$ lying over the fiber $F$ is less than $n_1$. Repeat this process until we get $g \in G$ such that any two points in $gZ \cap \mathcal O$ lie in different fibers of $\gamma$.
\end{proof}

\begin{remark} \label{two remark after prop 3.3}

(1) The closure $\overline{\mathcal O}$ of a $P_I^-$-orbit in $S=G/P$ is a Schubert variety, and  if $S$ has Picard number one, then $\overline {\mathcal O}$ has Picard number one, too (\cite{Ma}, \cite{Br}), and therefore, the exceptional locus  of $f$ is always nonempty (See Proposition \ref{exceptional locus is nonempty}).

(2) Applying Proposition \ref{transversality Lemma} to $Z$ and Borel group orbits in the complement of the open Borel orbit in $\mathcal O$, we may assume that $Z$ intersects only the open $B^-$-orbit in $\mathcal O$ at the beginning of the proof. In the rest of the proof we can still maintain that  the intersection points are contained in the open Borel group orbit, so that there is a translate of $Z$ satisfying conditions (1) - (3) in Proposition \ref{degeneration}.
\end{remark}

\begin{proposition} \label{exceptional locus is nonempty} Let $S=G/P$ be the rational homogeneous manifold of Picard number one and let $T(w)$ be an opposite Schubert variety of $S$ so that the stabilizer of $T(w)$ is $P_I^-$. Let $f:\widetilde {\mathcal O} \rightarrow \overline{\mathcal O}$ be the birational morphism associated with the open $P_I^-$-orbit $\mathcal O$ in $T(w)$. If $P_I^-$ has more than one orbits in $T(w)$, then the exceptional locus  of $f$ is nonempty.
\end{proposition}

\begin{proof}
The open $P_I^-$-orbit $\mathcal O$ is an $L_I$-homogeneous  fiber bundle over the fixed point set  $\mathcal H$ in $\mathcal O$.
By the same arguments as in the proof of Proposition \ref{the base has positive dimension},   $\mathcal H$ is of positive dimension.  If  $P_I^-$ has more than one orbit in $T(w)$, then   $\mathcal O$ is a nontrivial fiber bundle over   $\mathcal H$ and so is $\widetilde {\mathcal O}$. Hence $\widetilde{\mathcal O}$ has Picard number $\geq 2$.
 If $S$ has Picard number one, then so does $T(w)$ (\cite{Ma}).
 Since $\mathcal O$ is an open $P_I^-$-orbit in $T(w)$, its closure $\overline{\mathcal O}$ is $T(w)$, and thus $\overline{\mathcal O}$ has Picard number one.

 Suppose that the exceptional locus   of $f$ is empty. Then the proper birational morphism $f:\widetilde {\mathcal O} \rightarrow \overline{\mathcal O} $ is a biholomorphism because $\overline{\mathcal O}$ is normal. Thus $\widetilde {\mathcal O}$ and $\overline{\mathcal O}$ have the same Picard number, a contradiction.
Therefore,  the exceptional locus  of $f$ is nonempty.
\end{proof}

\begin{proposition} \label{degeneration} Let $S=G/P$ be the rational homogeneous manifold  and let $T(w)$ be an opposite Schubert variety of $S$ whose stabilizer in $G$ is $P_I^-$. Let $Z \subset S$ be a   reduced irreducible subvariety  of $S$ having homology class $[Z ] = r[S(w)]$ and
\begin{enumerate}
\item $Z $ intersects $ T(w)$ transversely and
\item $Z \cap T(w)$ is contained in $B^-.x_w$ and
\item $\gamma_*(Z  \cap T(w))=x_1 + \dots +x_r$
 with all $x_i$ being distinct,
\end{enumerate}
where  $\gamma$ is the projection map from $\mathcal O:=P_I^-.x_w $ to $ \mathcal H:=L_I.x_w$.
Consider the $\mathbb C^*$-action associated with $P_I^-$.
Let $\mathcal Z \subset S \times \mathbb  P^1$ be the closure of the union $\mathcal Z^0:=\bigcup_{t \in \mathbb C}t.Z \times \{t\}$ in $S \times \mathbb P^1$ and $\pi:\mathcal Z \rightarrow \mathbb P^1$ be the restriction of the second projection map $S \times \mathbb P^1 \rightarrow \mathbb P^1$.
 Write $\pi^{-1}(\infty)=Z_{\infty} \times \{\infty\}$.  Then $\mathcal Z$ is irreducible and $Z_{\infty} $ is $g_1S_0 + \dots + g_rS_0$, where $g_i \in G$ is such that $x_i=g_i.x_w$ for $i=1, \dots, r$.

\end{proposition}

\begin{remark}
It may hold under a weaker condition (2)' than the condition (2), that $Z \cap T(w)$ is contained in $\mathcal O$. For simplicity of the proof, we put the condition (2) instead of (2)'. As we remark after Proposition \ref{existence of perturbation}, this  will not cause any problem in later use.
\end{remark}

\begin{proof}
We will first consider the closure of $\mathcal Z^0 \cap (\mathcal U \times \mathbb P^1)$, where $\mathcal U:=w(U^-_P).x_w $  is the open big cell  in $S$ containing $x_w$.
Write $w(U^-_P)= \prod_{\alpha \in A}U_{\alpha}$ for a set $A$ consisting of roots.
Then
$$w(U^-_P)=\left(\prod_{\alpha \in A^+} U_{\alpha}\right)\left(\prod_{\beta \in A^0} U_{\beta} \right)\left(\prod_{\gamma \in A^-} U_{\gamma}\right),$$  where
\begin{eqnarray*}
A^+&=&\{\alpha \in A : \langle \alpha, \lambda \rangle >0 \}\\
A^0&=&\{\alpha \in A : \langle \alpha, \lambda \rangle =0 \}\\
A^-&=&\{\alpha \in A : \langle \alpha, \lambda \rangle <0 \}
\end{eqnarray*}
Choose coordinates $(z_{\alpha}, z_{\beta}, z_{\gamma})_{ \alpha \in A^+,\beta \in A^0,\gamma \in A^-}$ on $\mathcal U\simeq w(U^-_P)$ keeping this order.
Putting $n_{\alpha}=\langle \alpha, \lambda \rangle $ for $\alpha \in A^+$ and $n_{\gamma} = -\langle \gamma, \lambda \rangle$   for $\gamma \in A^-$,
we get that the $\mathbb C^*$-action induced by $\lambda$ can be expressed as:
$$t.(z_{\alpha}, z_{\beta}, z_{\gamma}) =(t^{n_{\alpha} }z_{\alpha}, z_{\beta}, t^{-n_{\gamma} }z_{\gamma})$$
for $t \in \mathbb C^*$.  The closure of $\{(0, z_{\beta}, z_{\gamma}): z_{\beta} \in \mathbb C^{|A^0|}, z_{\gamma} \in \mathbb C^{|A^-|}\}$ in $S$ is $T(w)$.  For each fixed $z_{\beta, 0} \in \mathbb C^{|A^0|}$, the closure of $\{(z_{\alpha}, z_{\beta,0}, 0): z_{\alpha} \in \mathbb C^{|A^+|}\}$ in $S$ is $gS(w)$, where $g \in L_I$ is such that $g.x_w$ has coordinates $(0, z_{\beta,0}, 0)$.

Let $Z \subset X$ be a subvariety satisfying the assumptions (1) -- (3). Write  $ Z \cap T(w)=\{p_1, \dots, p_r\}$. By (1) and (2), $p_i$ are all distinct and are contained in $\mathcal U$. Let $(p_{i, \alpha}, p_{i,\beta}, p_{i,\gamma})$ be the coordinates of $p_i$ for $i=1, \dots, r$. We will show that for each $p_i$ there is a neighborhood $V_i$ of $p_i$ in $Z$ such that $t.V_i$ converges to  $\{(z_{\alpha}, p_{i,\beta}, 0): z_{\alpha} \in \mathbb C^{|A^+|} \}$ as $t \rightarrow \infty$.

Fix $i$.
By (1) and (2), there is a neighborhood $V$ of $p_i$  in $Z \cap \mathcal U$ of the form
$$V=\left\{(z_{\alpha},  G_{\beta}(z_{\alpha}), H_{\gamma}(z_{\alpha})): (z_{\alpha})_{\alpha \in A^+ } \in V' \right\}$$
where  $V '$ is an open ball centered at $0$ in $\mathbb C^{|A^+ |}$, and $ G=(G_{\beta}),H=(H_{\gamma})$ are holomorphic maps defined on $V'$   with values in $\mathbb C^{|A^0|}$, $\mathbb C^{|A^-|}$ such that   $G_{\beta}(0)=p_{i,\beta}$.

Now
$$t.V=\left\{(t^{n_{\alpha}}z_{\alpha},   G_{\beta}(z_{\alpha}), t^{-n_{\gamma}}H_{\gamma}(z_{\alpha})): (z_{\alpha})_{\alpha \in A^+ } \in V' \right\}$$
for $t \in \mathbb C^*$. Put $y_{\alpha}= t^{n_{\alpha}}z_{\alpha}$. Then
$$t.V=\left\{( y_{\alpha},  G_{\beta}(t^{-n_{\alpha}}y_{\alpha}), t^{-n_{\gamma}}H_{\gamma}(t^{-n_{\alpha}}y_{\alpha})): (y_{\alpha})_{\alpha \in A^+ } \in t.V' \right\}.$$
As $t \rightarrow \infty$, $G_{\beta}(t^{-n_{\alpha}}y_{\alpha})$ tends to $G_{\beta}(0)$ and $t^{-n_{\gamma}}H_{\gamma}(t^{-n_{\alpha}}y_{\alpha})$ tends to zero, so that $t.V$ converges to
$$ \left\{( y_{\alpha},  p_{i,\beta}, 0): (y_{\alpha})_{\alpha \in A^+ } \in \mathbb C^{|A^+|}\right\}, $$
which is an open dense subset of  $g_iS(w)$.

   Consequently, the support of $Z_{\infty}$ contains the support of the cycle $g_1S(w) + \dots + g_rS(w)$.
  Since $\mathcal Z^0$ is irreducible, so is $\mathcal Z$. Thus $Z_{\infty}$ has the same homology class as $Z$ (Section 1.6 of \cite{Fu}). Therefore, $Z_{\infty}$ is equal to $g_1S(w) + \dots + g_rS(w)$.
\end{proof}

\section{Rigidity}

We review the geometric theory of uniruled projective manifolds modeled on varieties of minimal rational tangents   and its application to the homological rigidity of smooth Schubert varieties (\cite{HoM}, for an expository survey see \cite{Mk16}) and then we generalize the method to prove  the Schur rigidity.

\subsection{Varieties of minimal rational tangents}
Let $(X, \mathcal L)$ be a polarized projective manifold. A (parameterized) rational curve on $X$ is a nonconstant holomorphic map $f: \mathbb P^1 \rightarrow X$.
 A rational curve $f$ is said to be {\it free} if    the pull-back bundle $f^*TX$ on $\mathbb P^1$ is semipositive. A free rational curve $f$ such that the degree $f^*\mathcal L$ is minimum among all free rational curves is called a {\it minimal rational curve}.
  Let $\mathcal H$ be a connected component of $\text{Hom}(\mathbb P^1, X)$ containing a minimal rational curve and let $\mathcal H^0$  be the subset consisting of free rational curves. Then the quotient space $\mathcal K=\mathcal H^{0}/\Aut(\mathbb P^1)$ of (unparameterized) minimal rational curves  is  called
    a {\it minimal rational component}.

Fix a minimal rational component $\mathcal K$. By a minimal rational curve we implicitly  mean a minimal rational curve belonging to $\mathcal K$. For a general $x \in X$ the space $\mathcal K_x$ of minimal rational curves passing through $x$ is a projective manifold. Define a rational map from $\mathcal K_x$ to $\mathbb P(T_xX)$ by associating  a minimal rational curve immersed at $x$ to the tangent line at $x$, which is called the tangent map. The strict transformation $\mathcal C_x(X)$ of the tangent map is called the {\it variety of minimal rational tangents of} $X$ {\it at} $x$. The union of $\mathcal C_x(X)$ over general $x \in X$ forms a fibered space $\mathcal C(X)$ over $X$.

Let $S=G/P$ be a rational homogeneous manifold associated to a simple root. Then the Picard number of $S$ is one and the ample generator of the Picard group defines a  $G$-equivariant embedding of $S$ into a projective space $\mathbb P^N$. Lines $\mathbb P^1$ lying on $S$ are minimal rational curves, and we will choose the family $\mathcal K$ of lines lying on $S$ as our minimal rational component, so that the variety $\mathcal C_x(S)$ of minimal rational tangents of $S$ at any $x$  in $S$ is defined by the space of all tangent directions of lines lying on $S$ passing through $x$. If $S$ is associated to a long root, then $G$ acts on $\mathcal K$  transitively.  If $S$ is associated to a short root, then $\mathcal K$ has two $G$-orbits. In any case, by a general line we mean a line corresponding to  a point in   the open $G$-orbit  in $\mathcal K$, and by a general point in $\mathcal C_x(S)$ we mean the tangent direction of a general line.

Let $Z$ be an irreducible (possibly singular) subvariety of $S$ uniruled by lines contained in the smooth locus of $Z$.
By a general point in $Z$ we mean a point $z$ in the smooth locus of $Z$ such that    there is a line passing through $z$ contained in the smooth locus of $Z$ and the deformation of any line passing through $z$  covers an open dense subset of $Z$.    By the variety $\mathcal C_x(Z)$ of minimal rational tangents of $Z$ at a general $x \in Z$ we mean the variety of minimal rational tangents associated to  the family $\mathcal K_Z$  of lines lying on $Z$.

   A smooth Schubert variety $S_0$ of $S$ is uniruled by lines lying on $S_0$.  By a general point of $S_0$ we mean a point in the open orbit in $S_0$ of the stabilizer of $S_0$ in $G$. When there is a general line contained in $S_0$,  by  a general point in $\mathcal C_x(S_0)$ at a general point $x \in S_0$ we mean the tangent direction of a general line passing through $x$.
For  descriptions of $\mathcal C_x(S_0)$ see \cite{HoM}.

\subsection{Local characterizations} \label{sect: local characterizations}
Let $S=G/P$ be a rational homogeneous manifold associated to a   simple root, and $S_0 \subset S$ be a smooth Schubert variety.  Consider the following two conditions:
 \begin{enumerate}
 \item [(I)] at a general point $\alpha \in \mathcal C_x(S_0)$, for any $h \in P_x$ sufficiently
close to the identity element $e\in P_x$ and satisfying
$T_{\alpha}\left(h\mathcal C_x(S_0)\right) =
T_{\alpha}\left(\mathcal C_x(S_0)\right)$ we must have $h\mathcal C_x(S_0) = \mathcal C_x(S_0)$;
 \item [(II)] any local deformation    of $\, \mathcal C_x( S_0)$ in $\mathcal C_x( S )$ is induced by the action of $P_x$,
 \end{enumerate}
 where  $P_x$ is   the isotropy of $G$ at a general point  $x \in S_0$.

\begin{proposition} [Proposition 3.2 and Proposition 3.6 of \cite{HoM}]  \label{general case - inductive step}
Let $S=G/P$ be a rational homogeneous manifold associated to a   simple root, and $S_0 \subset S$
be a smooth Schubert variety.
Assume that   $\mathcal C_x(S_0)$ satisfies {\rm(I)} and {\rm(II)} at a general point $x \in S_0$.
Then,  the following holds true.
  \begin{enumerate}
 \item[\rm (1)]
    If  a  smooth subvariety  $Z$ of $S$ is uniruled by lines of $S$ lying on $Z$ and  contains $x$ as a general point with  $\mathcal C_x(Z) = \mathcal C_x( S_0)$, then $S_0$ is contained in $Z$.
  \item[\rm (2)]
Any local deformation of $S_0$ in $S$ is induced by the action of $G$.
 \end{enumerate}

\end{proposition}

 Proposition \ref{general case - inductive step} (2), together with Proposition 2.2 of \cite{HoM}, implies that (I) and (II) are  sufficient conditions for a smooth Schubert variety to be homologically rigid.

\begin{theorem} [\cite{HoM}] \label{local rigidity theorem}
Let $S=G/P$ be a rational homogeneous manifold associated to a   simple root, and $S_0 \subset S$ be a smooth Schubert variety.
Assume that   $\mathcal C_x(S_0)$ satisfies {\rm(I)} and {\rm(II)} at a general point $x \in S_0$.
 Then any subvariety $Z \subset S$ having homology class $[Z] =   [S_0]$  is $gS_0$ for some $g \in G$.
\end{theorem}

We will modify the proof of Proposition \ref{general case - inductive step} to get sufficient conditions for a smooth Schubert variety to be Schur rigid (Theorem \ref{sufficient condition for Schur rigidity}). Proposition \ref{existence of perturbation} and Proposition \ref{degeneration} will replace arguments in the proof of Proposition 2.2 of \cite{HoM}. In what follows for a smooth Zariski open subset $W$ of an irreducible projective subvariety  $\Sigma \subset \mathbb P^N$, $N \ge 2$, we say that $W$ is uniruled by lines to mean that there exists a projective line $\ell$ lying on $W$ with semipositive normal bundle.  Equivalently, $W$ is uniruled by lines if and only if the union of lines lying on $W$ covers a nonempty Zariski open subset of $W$.

\begin{proposition} \label{characterization-modification}
Let $S=G/P$ be a rational homogeneous manifold associated to a   simple root, and $S_0 \subset S$
be a smooth Schubert variety.
Assume that   $\mathcal C_x(S_0)$ satisfies {\rm(I)} and {\rm(II)} at a general point $x \in S_0$ as in Proposition \ref{general case - inductive step}. Then, the following holds true.

\begin{enumerate}
\item [{\rm(1)}] Let $Z$ be a reduced and irreducible {\rm(}possibly singular{\rm)}  subvariety of $S$ uniruled by  lines contained in the smooth locus  of $Z$. If $Z$ contains $x$ as a general point with  $\mathcal C_x(Z) = \mathcal C_x( S_0)$, then $S_0$ is contained in $Z$.

\item [{\rm(2)}] Let $\mathcal Z \subset S \times \mathbb C$ be an irreducible {\rm(}possibly singular{\rm)}  subvariety   with the restriction $\pi : \mathcal Z \rightarrow \mathbb C$ of the second projection map $S \times \mathbb C \rightarrow \mathbb C$. If  there is a reduced and smooth quasi-projective subvariety $\widehat S_0 \times \{0\}$ of   $\pi ^{-1}(0) \cap (S_0 \times \{0\})$  with $\dim \widehat{S}_0=\dim S_0$ which is  uniruled by lines contained in $\widehat S_0 \times \{0\}$, then for  $t \in \mathbb C$ sufficiently small, there is $g_t \in G$ such that $g_tS_0 \times \{t\}$ is contained in $\pi^{-1}(t)$. 
\end{enumerate}
\end{proposition}

\begin{proof} (1)    By applying the same arguments as in the proof of Proposition \ref{general case - inductive step} (1) (Proposition 3.2 of \cite{HoM}), using the assumption (I) and (II) and deformation theory of minimal rational curves, we get that $S_0$ is contained in $Z$.

(2)   From the smoothness of $   \widehat{ S}_0$ it follows that there is a restriction $\widehat{\pi} : \widehat{\mathcal Z} \rightarrow \mathbb C$ of $\pi:\mathcal Z \rightarrow \mathbb C$ to a smooth quasi projective Zariski dense open subset $\widehat{\mathcal Z}$ of $\mathcal Z$ with $\widehat{\pi}$ being smooth and ${\widehat{\pi}}^{-1}(0)=\widehat{S}_0\times \{0\}$, and $ {\widehat{\pi}}^{-1}(t)$ is contained in the smooth locus of $\pi^{-1}(t)$ for any $t$ (See the proof of Lemma 1.1 of \cite{MkZh}).
 For $\epsilon >0$ sufficiently small there exists a holomorphic section $\sigma: \Delta(\epsilon) \rightarrow \widehat{\mathcal Z}$ of $\widehat{\pi}:\widehat{\mathcal Z} \rightarrow \mathbb C$ over $\Delta(\epsilon)$ such that $\sigma(0)$ is a general point of $S_0$.
 By Kodaira stability \cite{Kd63} any line contained in ${\widehat{\pi}}^{-1}(0)$ passing through $\sigma(0)$ can be deformed to a line in ${\widehat{\pi}}^{-1}(t)$ passing through $\sigma(t)$. By putting together all such lines we get
 $\bigcup_{t \in \Delta(\epsilon)}\mathcal C^0_{\sigma(t)}({\widehat{\pi}}^{-1}(t)) \subset \bigcup_{t \in \Delta(\epsilon)}\left(\mathcal C_{\sigma(t)}(S)\times \{t\}\right)$, where $\mathcal C^0_{\sigma(t)}(\widehat{\pi}^{-1}(t))$ is a Zariski open subset of some subvariety of $\mathcal C_{\sigma(t)}(S) \times \{t\}$.
%
 Taking topological closure in $ \bigcup_{t \in \Delta(\epsilon)}\left(\mathcal C_{\sigma(t)}(S)\times \{t\}\right)$ we obtain $\bigcup_{t \in \Delta(\epsilon)}\mathcal C _{\sigma(t)}(\widehat{\pi}^{-1}(t)) \subset \bigcup_{t \in \Delta(\epsilon)}\left(\mathcal C_{\sigma(t)}(S)\times \{t\}\right)$. Since $\mathcal C _{\sigma(0)}({\widehat{\pi}}^{-1}(0))=\left(\mathcal C_{\sigma(t)}(S)\times \{0\}\right)$ is smooth,    $\bigcup_{t \in \Delta(\epsilon)}\mathcal C _{\sigma(t)}({\widehat{\pi}}^{-1}(t))\rightarrow \mathbb C$ can be regarded as a regular family of submanifolds of $\mathcal C_{x}(S)$. By the assumption (II), there is $g_t \in G$ such that  $\mathcal C _{\sigma(t)}({\widehat{\pi}}^{-1}(t))=\mathcal C _{\sigma(t)}(g_tS_0)\times \{t\}$. By (1), $g_tS_0 \times \{t\}$ is contained in $\pi^{-1}(t)$.
\end{proof}

\begin{remark}

\label{Rem:characterization-modification}
\quad As mentioned in the Introduction, in order to adapt the proof of Proposition 4.3 to a singular Schubert variety $S_0 \subset S$ we need at least to have an ample supply of minimal rational curves on the smooth locus of $S_0$ so that Kodaira Stability Theorem can be applied. However, there exist singular Schubert varieties $S_0$ on certain $S$ such that all projective lines on $S$ lying on $S_0$ must pass through the singular locus of $S_0$.  This is the case, for instance, when $S_0$ is the Lagrangian Grassmannian of rank $\ge 2$ and $S_0 \subset S$ is the subvariety swept out by minimal rational curves on $S$ passing through a fixed base point $x_0 \in S$.  In this case $S_0 \subset S$ is a Schubert variety with a unique isolated singularity $x_0$, and all minimal rational curves on $S$ lying on $S_0$ must pass through the base point $x_0$.

\end{remark}

\begin{proposition} \label{nonexistence}
Let $S=G/P$ be a rational homogeneous manifold associated to a   simple root, and $S_0 \subset S$ be a smooth Schubert variety.
Assume that   $\mathcal C_x(S_0)$ satisfies {\rm(I)} and {\rm(II)} at a general point $x \in S_0$, and $S_0$ intersects $gS_0$ in codimension 2 for any $g \in G$.
 Then there is no reduced irreducible subvariety $Z \subset S$ having homology class $[Z] = r [S_0]$ for any $r \geq 2$.
\end{proposition}

\begin{proof} Assume that there exists a   reduced irreducible subvariety $Z $ of $S$ having homology class $[Z] = r[S_0]$ for some $r \geq 2$. Let $T_0$ be the Schubert variety opposite to $S_0$ and  $P_I^-$ be the stabilizer of $T_0$ in $G$. Denote by  $\gamma:\mathcal O \rightarrow \mathcal H$   the projection map from the open $P_I^-$-orbit $\mathcal O$ in $T_0$ to the fixed point set $\mathcal H$ in $\mathcal O$ of the $\mathbb C^*$-action associated to $P_I^-$. Then by Proposition \ref{existence of perturbation} and Remark \ref{two remark after prop 3.3},
there exists a reduced irreducible subvariety $Z' $ of $S$ having homology class $[Z'] = r[S_0]$ and
\begin{enumerate}
\item $Z'$ intersects $T_0$ transversely and
\item $Z' \cap T_0$ is contained in $B^-.x_w $ and
\item $\gamma_*(Z' \cap T_0)=x_1 + \cdots + x_r$, where all $x_i$ are distinct.
\end{enumerate}
As in Proposition \ref{degeneration}, take the closure $\mathcal Z$ in $S \times \mathbb P^1$ of the union $\bigcup_{t \in \mathbb C^*} t.Z' \times \{t\}$.
By Proposition \ref{degeneration}, $\mathcal Z$ is irreducible and $\pi^{-1}(\infty) = (g_1S_0 + \dots + g_r S_0) \times \{\infty\}$.
By the assumption that $S_0$ intersects $gS_0$ in codimension 2, up to the action of $G$, there is a smooth quasi-projective subvariety $\widehat{S}_0 \times \{\infty\}$ of $\pi^{-1}(\infty) \cap (S_0 \times \{\infty\})$ with $\dim \widehat{S}_0 = \dim S_0$ which is uniruled by lines in $\widehat{S}_0 \times \{\infty\}$. Applying Proposition \ref{characterization-modification} (2) with $t$ replaced by $\frac{1}{t}$, for $t$ sufficiently large, we have $t.Z' = g_tS_0$ for some $g_t \in G$. But, then, $[Z']=[S_0]$ while $[Z']=r[S_0]$ with $r \geq 2$, a contradiction. Consequently, there is no  reduced irreducible subvariety $Z  $ of $S$ having homology class $[Z] = r[S_0]$ for some $r \geq 2$.
\end{proof}

\begin{theorem} \label{sufficient condition for Schur rigidity}
Let $S=G/P$ be a rational homogeneous manifold associated to a   simple root, and $S_0 \subset S$ be a smooth Schubert variety.
Assume that   $\mathcal C_x(S_0)$ satisfies {\rm(I)} and {\rm(II)} at a general point $x \in S_0$, and $S_0$ intersects $gS_0$ in codimension 2 for any $g \in G$. Then for any subvariety $Z \subset S$ with homology class $[Z] =r[S_0]$ we have $Z=n_1g_1S_0 + \dots + n_sg_sS_0$, where $g_i \in G$ and $n_i \in \mathbb Z_{+}$.

\end{theorem}

\begin{proof}
Consider the decomposition $Z=m_1Z_1 + \dots + m_lZ_l$  of $Z$ by its irreducible components, where $Z_i$ are reduced and $m_i \in \mathbb Z_{+}$ for all $i$. From $[Z]=r[S_0]$ it follows that $[Z_i]=r_i[S_0]$ for some $r_i$. By Proposition \ref{nonexistence} we have $r_i=1$. By  Theorem \ref{local rigidity theorem}, $Z_i$ is $g_iS_0$ for some $g_i \in G$. Therefore, $Z=m_1g_1S_0 + \cdots + m_lg_lS_0$.
       \end{proof}

\subsection{Intersections of translates}

 It remains to confirm the validity of the conditions in Theorem \ref{sufficient condition for Schur rigidity} for non-linear smooth Schubert varieties $S_0 \subset S$.

Let  $\{ \alpha_1, \dots, \alpha_n\}$ be the system of simple roots of a simple Lie group $L$ and
 let $\{ \varpi_1, \dots, \varpi_{n}\}$ be the system of fundamental weights  of $L$. Let $V_L(\varpi_i)$ denote the irreducible $L$-representation space with the highest weight $\varpi_i$. Take a highest weight vector $v_i$ in $V_L(\varpi_i)$ for $i=1, \dots, n$. We denote the closure   of the $L$-orbit of $[v_i + v_j]$ in $\mathbb P(V_L(\varpi_i) \oplus V_L(\varpi_j))$   by $(L, \alpha_i, \alpha_j)$.


\begin{proposition} [Proposition 3.7 of \cite{HoM}, Theorem 1.3 of \cite{HoK}] \label{Classification of smooth Schubert varieties}

Let $S=G/P$ be a rational homogeneous manifold associated to a simple root. Then a non-linear smooth Schubert variety $S_0$ of $S$ is either a homogeneous submanifold associated to a subdiagram of $S$ or one of the following.

 \begin{enumerate}
 \item[\rm(1)] $S_0=(C_n, \alpha_{i+1}, \alpha_i)$ and $S=(C_m, \alpha_k)$, $2 \leq n \leq m$ and $1 \leq i \leq n-1$ and $m-k=n-i$;
 \item[\rm(2)] $S_0=(C_2, \alpha_2, \alpha_1)$ and $S=(F_4, \alpha_3)$;
 \item[\rm(3)] $S_0 =(B_3, \alpha_2, \alpha_3)$ and $S=(F_4, \alpha_3)$.
 \end{enumerate}

\end{proposition}

\begin{example} We describe embedding of $S_0$ in $S$ for the pairs $(S,S_0)$ in the list (1) - (3) of Proposition \ref{Classification of smooth Schubert varieties}.

(1)
(\cite{Pa})
Take an isotropic basis $\{ e_1 , \cdots, e_{2n}\}$ of $\mathbb C^{2n}$ and extend it to an isotropic basis $\{ e_0, e_1, \cdots, e_{2n}, e_{2n+1} \}$ of $\mathbb C^{2n+2}$.   The embedding $V_{C_n}(\varpi_1)=\mathbb C^{2n} \rightarrow V_{C_{n+1}}(\varpi_1)=\mathbb C^{2n+2}$  induces an embedding  $V_{C_n}(\varpi_{i+1}) \subset \wedge^{i+1}\mathbb C^{2n} \rightarrow V_{C_{n+1}}(\varpi_{i+1})\subset \wedge^{i+1}\mathbb C^{2n+2}$ and an embedding $V_{C_n}(\varpi_i) \subset \wedge^{i }\mathbb C^{2n} \rightarrow V_{C_{n+1}}(\varpi_{i+1}) \subset \wedge^{i+1}\mathbb C^{2n+2}$ given by   $v_1 \wedge \cdots \wedge v_i \mapsto v_1 \wedge \cdots \wedge v_i \wedge e_0$. Combining these two embeddings we get an embedding
\begin{eqnarray*}
V_{C_n}(\varpi_{i+1}) \oplus V_{C_n}(\varpi_i) &\rightarrow& V_{C_{n+1}}(\varpi_{i+1}) \\
(e_1 \wedge \cdots \wedge e_i \wedge e_{i+1} , e_1 \wedge \cdots \wedge e_i )  &\mapsto&  e_1 \wedge \cdots \wedge e_i \wedge e_{i+1} + e_1 \wedge \cdots \wedge e_i \wedge e_{0}  \\
&& = e_1 \wedge \cdots \wedge e_i \wedge (e_0 + e_{i+1}).
\end{eqnarray*}
By taking the closure of $Sp(2n)$-orbit of $ e_1 \wedge \cdots \wedge e_i \wedge (e_0 + e_{i+1})$ in  $V_{C_{n+1}}(\varpi_{i+1})$, we get  an embedding of $S_0=(C_n, \alpha_{i+1}, \alpha_i)$ into $S=(C_{n+1}, \alpha_{i+1})$. Geometrically, $S_0$ is the space of isotropic $(i+1)$-subspaces of $\mathbb C^{2n+2}$ contained in the subspace  $H:=\mathbb C e_0 + \mathbb C^{2n}$ of $\mathbb C^{2n+2}$.

(2) After taking the composition of the embedding of $(C_2, \alpha_2, \alpha_1)$ into $(C_3, \alpha_2)$ as in the case of (1) and the  embedding of $(C_3, \alpha_2)$ into $(F_4, \alpha_3)$ as a homogeneous submanifold, we get an embedding of $(C_2, \alpha_2, \alpha_1)$ into $(F_4, \alpha_3)$

(3)
We recall some facts on the projective geometry of the rational homogeneous manifold $\mathbb {OP}_0^2$ of type $(F_4, \alpha_4)$ (Proposition 6.6 and Proposition 6.7 of   \cite{LM}).
 Let $\mathcal J_3(\mathbb O)$ be the space of $3 \times 3$ $\mathbb O$-Hermitian symmetric matrices
$$ \mathcal J_3(\mathbb O) = \left\{ B= \left(
 \begin{array}{ccc}
 r_1 &\overline{x_3} & \overline{x_2} \\
 x_3 & r_2 & \overline{x_1} \\
 x_2 & x_1 & r_3
 \end{array}\right), r_i \in \mathbb C, x_j \in \mathbb O \right\}   $$
 and let $\mathcal J_3(\mathbb O)_0:=\{ B\in \mathcal J_3(\mathbb O) : \text{tr} B=0\}$.
Then we have   $V_{E_6}(\omega_1)=\mathcal J_3(\mathbb O)$ and $V_{F_4}(\omega_4)=\mathcal J_3(\mathbb O)_0$ and
\begin{eqnarray*}
   \mathbb {OP}_0^2  :=\mathbb P\{ A \in \mathcal J_3(\mathbb O)_0 : A^2=0\} \subset \mathbb P(V_{F_4}(\omega_4))
 \end{eqnarray*}
 is the rational homogeneous manifold of type $(F_4, \alpha_4)$.

Let $ A \in \mathcal J_3(\mathbb O)_0 $ be such that $A^2=0$. Then the affine tangent space $\widehat{T}_{[A]}\mathbb {OP}_0^2$ of $\mathbb {OP}_0^2$ at $[A]$ has a filtration
$$\mathbb C A \subset  \widehat{T}_{[A],1}\mathbb {OP}_0^2\subset \widehat{T}_{[A]}\mathbb {OP}_0^2$$
invariant under the isotropy group of $G$ at $[A]$,
where
 \begin{eqnarray*}
 \widehat{T}_{[A]}\mathbb {OP}_0^2&=& \{B \in \mathcal J_3(\mathbb O)_0 : AB+BA=0\}\\
 \widehat{T}_{[A],1}\mathbb {OP}_0^2&=& \{B \in \mathcal J_3(\mathbb O)_0 : AB=0\}.
 \end{eqnarray*}
 Furthermore,  as a representation of $Spin (7)$, $\widehat{T}_{[A]}\mathbb {OP}_0^2$ is decomposed as $\mathbb C A \oplus V_{B_3}(\omega_3) \oplus V_{B_3}(\omega_1)$ and  $ \widehat{T}_{[A],1}\mathbb {OP}_0^2  $ is decomposed as $\mathbb C A \oplus V_{B_3}(\omega_3)$. Therefore,  we have an embedding of $\mathbb C A \oplus V_{B_3}(\omega_3)$ into $V_{F_4}(\omega_4)$.

 From this   embedding
 we get an  embedding of $V_{B_3}(\omega_3)$ into $V_{F_4}(\omega_3)$   defined by
 $$B \in V_{B_3}(\omega_3) \mapsto A \wedge B \in V_{F_4}(\omega_3) \subset \wedge^2 (V_{F_4}(\omega_4))$$
   and an  embedding  of  $V_{B_3}(\omega_2) \subset  \wedge^2( V_{B_3}(\omega_3))$ into $V_{F_4}(\omega_3) \subset  \wedge^2(V_{F_4}(\omega_4) ) $.
   Combining these two embeddings we get an embedding
 \begin{eqnarray*}
 V_{B_3}(\omega_2) \oplus V_{B_3}(\omega_3) &\rightarrow& V_{F_4}(\omega_3) \subset \wedge^2 V_{F_4}(\omega_4) \\
 (C \wedge B,  B ) \,\,\,\, \qquad &\mapsto& C \wedge B + A \wedge B =(C+A) \wedge B.
 \end{eqnarray*}
 By taking the closure of the $Spin(7)$-orbit of $(C+A) \wedge B$ in $V_{F_4}(\omega_3) $ we get an  embedding of $S_0 =(B_3, \alpha_2, \alpha_3)$ into $S =(F_4, \alpha_3)$. Geometrically, $S_0$ is the space of $2$-subspaces $\langle C_1, C_2\rangle$ of $\mathcal J_3(\mathbb O)_0$  contained in the subspace $H:=\{B \in \mathcal J_3(\mathbb O)_0: AB=0\}$ of $\mathcal J_3(\mathbb O)_0$, where $C_1^2=C_2^2=C_1C_2=0$.
\end{example}

\begin{proposition} \label{I II codimension}

Let $S=G/P$ be a rational homogeneous manifold associated to a   simple root, and $S_0 \subset S$ be a smooth non-linear  Schubert variety.
Then   $\mathcal C_x(S_0)$ satisfies {\rm(I)} and {\rm(II)} at a general point $x \in S_0$, and $S_0$ intersects $gS_0$ in codimension 2 for any $g \in G$ with $S_0 \not=gS_0$.
\end{proposition}

\begin{proof}
By Proposition \ref{Classification of smooth Schubert varieties}, a non-linear smooth Schubert variety $S_0$ of $S$ is either a homogeneous submanifold associated to a subdiagram of $S$ or one of the followings.

 \begin{enumerate}
 \item[\rm(1)] $S_0=(C_n, \alpha_{i+1}, \alpha_i)$ and $S=(C_m, \alpha_k)$, $2 \leq n \leq m$ and $1 \leq i \leq n-1$ and $m-k=n-i$;
 \item[\rm(2)] $S_0=(C_2, \alpha_2, \alpha_1)$ and $S=(F_4, \alpha_3)$;
 \item[\rm(3)] $S_0 =(B_3, \alpha_2, \alpha_3)$ and $S=(F_4, \alpha_3)$.
 \end{enumerate}
In   the proof of Proposition \ref{Classification of smooth Schubert varieties} we have already  proved that properties (I) and (II) hold for a non-linear smooth Schubert variety   $S_0$ (See Proposition 3.3 and Proposition 3.5 of \cite{HoM} for a homogeneous submanifold associated to a subdiagram of $S$, Proposition 4.3 and Lemma 4.4 of \cite{HoM} for (1), and Proposition 3.4 of \cite{HoK} for (2) and (3)). This completes the proof of the first statement.

To prove the second statement, we first consider the case when  $S_0$ is
a   homogeneous submanifold associated to a subdiagram $\mathcal D_0$ of  $\mathcal D(S)$.
Then the stabilizer of $S_0$ in $G$ is the parabolic subgroup $P_I$, where $\Lambda$ is the set  of simple roots in $\mathcal D(S)- \mathcal D_0$  which are adjacent to $\mathcal D_0$ and  $I$ is  the complement of $\Lambda$ in the set of simple roots in $\mathcal D(S)$. Therefore, we have an isomorphism $\{gS_0:g \in G\} \simeq G/P_I$.

Let $s_{\gamma}$ denote  the element in $\mathcal W^{P_I}$ given by the simple reflection with respect to $\gamma$.
We claim that if $s_{\gamma}S_0 \cap S_0$ has dimension $\leq \dim S_0 -2$ for any simple root $\gamma$ in $\Lambda$, then $gS_0 \cap S_0$ has dimension $\leq \dim S_0 -2$ for any $g \in G$ with $S_0 \not=gS_0$.

To prove the claim, put
\begin{eqnarray*}
G_1&:=&\{ g \in G : \dim (S_0 \cap gS_0) \geq \dim S_0 -1 \} \\
\mathcal K_1&:=&\{gP_I \in G/P_I: g \in G_1\}.
\end{eqnarray*}
Then $G_1$ contains the stabilizer $P_I $ of $S_0$.  Since $B$ stabilizes $S_0$, for $b,b'\in B$ and $g \in G$, we have
$S_0 \cap bgb'S_0 = bS_0 \cap bgb'S_0 = b(S_0 \cap gb'S_0) = b(S_0 \cap gS_0) $ so that $\dim(S_0 \cap bgb'S_0) = \dim(S_0 \cap gS_0)$.
As a consequence, $G_1$ is a union of double $B$-cosets.

    If $  G_1$  contains   $P_I$  properly,    then $ \mathcal K_1$ is a positive-dimensional  subvariety of $G/P_I$ which is  invariant by $B$.
    By the Bruhat decomposition of $G/P_I$,  the closure of   a  $B$-orbit in $G/P_I$ of dimension $k$ has at least one $B$-orbit of dimension $j$ for each  $j \leq  k$. Hence there is a $B$-orbit  in $\mathcal K_1$ of dimension one.
 A one-dimensional $B$-orbit in  $G/P_I$ corresponds to a simple root  in $\Lambda$, and thus there is $\gamma \in \Lambda$ such that  $\dim (S_0 \cap s_{\gamma}S_0) \geq \dim S_0 -1$.    This completes the proof of the claim.
 In the remaining part of the proof we will show that $s_{\gamma}S_0 \cap S_0$ has dimension $\leq \dim S_0 -2$ for any simple root $\gamma$ in $\Lambda$. \\ 

Let $x_0$ denote the base point of $G/P$ with the isotropy group $P$ and let $\alpha_k$ be the simple root associated to $P$. Then the tangent space  $T_{x_0}S_0$ is generated by the root vectors of roots  with zero coefficients   in the simple roots in $\Lambda$ and a negative  coefficient in $\alpha_k$. The isotropy action of  $s_{\gamma}$ on $T_{x_0}S$ is given by mapping an element  in the root space $\frak g_{\alpha}$ to an element in  the root space $\frak g_{s_{\gamma}(\alpha)}$.
 If $S_0$ is not linear, then for any $\gamma \in \Lambda$, there are at least two roots   $\alpha$ with $\frak g_{\alpha} \subset T_{x_0}S_0$ such that $\langle \alpha, \gamma \rangle \not=0$. 
  Furthermore, for $\gamma \in \Lambda$ with $\langle \alpha, \gamma \rangle \not=0$, we have  $\frak g_{s_{\gamma}(\alpha)} =\frak g_{\alpha - \langle \alpha, \gamma \rangle \gamma} \not\subset T_{x_0}S_0$.  Thus, for any $\gamma \in \Lambda$,
$T_{x_0}S_0 \cap T_{x_0}(s_{\gamma} S_0)$ has dimension $\leq \dim T_{x_0}S_0-2$ and $s_{\gamma}S_0 \cap S_0$ has dimension $\leq \dim S_0 -2$.

If $(S_0, S)$ is of type (1)  in Proposition \ref{Classification of smooth Schubert varieties}, then $S_0$ consists of isotropic $k$-subspaces of $V=\mathbb C^{2n+2}$ which are contained in a hyperplane $H$ of $V$. Thus $S_0 \cap gS_0$  consists of isotropic $k$-subspaces of $V=\mathbb C^{2n+2}$ which are contained in the intersection $H \cap gH$. If $S_0 \not=gS_0$, then $H \cap gH$ is a proper subspace of $H$ and thus we have $\dim (S_0 \cap gS_0) \leq \dim S_0-2$.
The proof is similar when $(S_0, S)$ is of type (2)  in Proposition \ref{Classification of smooth Schubert varieties}

If $(S_0,S)$ is of type (3) in Proposition \ref{Classification of smooth Schubert varieties}, then $S_0$ consists of $C_1 \wedge C_2$, where $C_1, C_2 \in H:=\{ B \in \mathcal J_3(\mathbb O)_0 : AB=0 \}$  
are such that $C_1^2=C_2^2=C_1C_2=0 $.  If $S_0 \not= gS_0$, then $H \cap gH =\{ B \in \mathcal J_3(\mathbb O)_0  : AB=(gA)B=0\}$ is a proper subspace of $H$ and thus we have $\dim (S_0 \cap gS_0) \leq \dim S_0-2$.
\end{proof}

\vskip 10 pt
\noindent {\it Proof of Theorem \ref{main theorem I}.}
 From Theorem \ref{sufficient condition for Schur rigidity} and Proposition \ref{I II codimension} it follows that for any non-linear smooth Schubert variety $S_0$ of a rational homogeneous manifold $S$ of Picard number one, the pair $(S,S_0)$ is Schur rigid.
\qed \\

\subsection{Maximal linear spaces}
For obvious reasons linear Schubert variety is not Schur rigid if  it is not a maximal linear space. In this section we will show that a maximal linear Schubert variety is Schur rigid with some trivial exceptions. \\

Let $(S,S_0)$ be a pair consisting of a rational homogeneous manifold $S = G/P$ of Picard number one and a smooth Schubert variety $S_0 \subset S$.  We denote by $0 \in S_0$ a reference point lying on the unique open $B$-orbit.  Recall that $(S,S_0)$ is {\it Schubert rigid} if and only if the following holds.  Any complex submanifold $Z \subset W$ of some connected open subset $W \subset S$ must necessarily be an open subset of a translate $gS_0$ of $S_0$ by some $g \in G$ whenever for a point $x$ belonging to some nonempty open subset $U \subset W$ there exists some $g \in \text{\rm Aut}(S)$, where $g$ depends on $x$, such that $g(0) = x$ and such that $dg(T_0S_0) = T_xZ$. The latter condition may also be formulated more geometrically by stating that
$Z$ is tangent at every point $x \in U$ to some translate $gS_0$, $g \in G$, at the point $g(0) = x$. Here ``open'' means ``open with respect to the complex topology''.  In what follows we are concerned solely with the cases where $S_0 \subset S$ is a maximal linear subspace.
From the discussion in the Introduction it follows that Schubert rigidity is a necessary condition for Schur rigidity.

 \begin{theorem} [cf. Theorem 1.3 of \cite{HoPa}] \label{thm Hong park}  Let $S=G/P$ be a rational homogeneous manifold associated to a simple root and let $S_0$ be a maximal linear space in $S$. Then $(S,S_0)$ is Schubert rigid   whenever $(S, S_0)$ is not of the type belonging to any of the classes given by
 \begin{enumerate}
 \item $S=(B_{n}, \alpha_k)$ $(k \leq n-2)$ and $S_0 =\mathbb P^{n-k}${\rm ;}
 \item $S=(C_{n}, \alpha_{n})$ and $S_0=\mathbb P^1${\rm ;}
 \item $S=(F_4, \alpha_1)$ and $S_0=\mathbb P^2${\rm ;}
 \item $S=(G_2, \alpha_2)$ and $S_0=\mathbb P^1$.
 \end{enumerate}

\noindent
Furthermore, when $(S,S_0)$ is of the type belonging to any one of the classes (1)--(3) in the above, $(S,S_0)$ is not Schubert rigid.  More precisely, there exists a nonlinear irreducible projective subvariety $Z \subset S$ such that $Z$ is tangent at every smooth point to some maximal linear subspace $gS_0$, $g \in G$, of $S$.
 \end{theorem}

Theorem \ref{thm Hong park} is a restatement of Theorem 1.3 of \cite{HoPa} with   corrections.
    When $(S,S_0)=((B_n, \alpha_{k}), \mathbb P^{n-k})$ for $k=n-1$,  $S_0$ is not a maximal linear space and this case was included mistakenly in the list of exceptional cases in Theorem 1.3 of \cite{HoPa}.
 When $(S,S_0) =((G_2, \alpha_2), \mathbb P^1)$, the arguments in the proof of Lemma 5.1 of \cite{HoPa} do not work.

 The method of constructing counterexamples explained in p. 2354 of \cite{HoPa} works for the cases (1)--(3) in Theorem \ref{thm Hong park}  but does not work for the case (4) in Theorem \ref{thm Hong park}. It remains  open whether this is Schubert rigid  or not.

 We note that Corollary 1.2 of \cite{MkZh}, which is logically a consequence of Main Theorem of \cite{MkZh} and Theorem 1.3 of \cite{HoPa}, should also be amended accordingly.

\begin{example} \label{example maximal linear spaces in F4 type} 

  The odd Spinor variety $\mathbb S_{B_3}$, the rational homogeneous manifold of type $(B_3, \alpha_3)$, has two types of $\mathbb P^3$'s: one is a homogeneous submanifold  associated to a subdiagram  of the Dynkin diagram of $\mathbb S_{B_3}$ and the other is not. We denote the first by $\mathbb P^3_{B_2}$ and the second by $\mathbb P^3_{A_3}$.

In fact, the odd Spinor variety $\mathbb S_{B_3}$ is biholomorphic to the even Spinor variety $\mathbb S_{D_4}$, the rational homogeneous manifold of type $(D_4, \alpha_4)$.
There are two types of $\mathbb P^3$'s in $\mathbb S_{D_4}$,  each of which is a  homogeneous submanifold    associated to a  subdiagram  $\mathcal D_0$  of the Dynkin diagram $\mathcal D$  of $\mathbb S_{D_4}$ (Theorem 4.9 of \cite{LM}).
In both cases, the subdiagram $\mathcal D_0$ is of type $A_3$, but only the first one still is associated to a subdiagram of the Dynkin diagram of $\mathbb S_{B_3}$, which is of type $B_2$.  Under the biholomorphism from $\mathbb S_{B_3}$ to $\mathbb S_{D_4}$, the first corresponds to $\mathbb P^3_{B_2}$ and the second  corresponds to $\mathbb P^3_{A_3}$.

   The odd Spinor variety $\mathbb S_{B_3}$ is a homogeneous submanifold of the rational homogeneous manifold $S$ of type $(F_4 ,\alpha_3)$. Thus $\mathbb P^3_{B_2}$ and $\mathbb P^3_{A_3}$ are linear spaces in $S$, too, and both of them are maximal.

  The odd Spinor variety $\mathbb S_{B_3}$ is isomorphic to  the space of non-generic lines on the rational homogeneous manifold $\mathbb {OP}_0^2$ of type $(F_4, \alpha_4)$ passing through a point (Proposition 6.5 of \cite{LM}). Thus $\mathbb P^3_{B_2}$ and $\mathbb P^3_{A_3}$ in $\mathbb S_{B_3}$  correspond  to two types of $\mathbb P^4$'s in $\mathbb {OP}_0^2$, the first is not maximal and the second is maximal. We denote the second by $\mathbb P^4_{A_4}$.
\end{example}

\begin{proposition} \label{F4 not homologically rigid} Let $(S,S_0)$ be one of the following:
\begin{enumerate}
\item $S=(F_4, \alpha_3)$ and $S_0=\mathbb P^3_{B_2}$
\item $S=(F_4, \alpha_3)$ and $S_0 =\mathbb P^3_{A_3}$
\item $S=(F_4, \alpha_4)$ and $S_0=\mathbb P^4_{A_4}$,
\end{enumerate}
 where $\mathbb P^3_{B_2}$ and $\mathbb P^3_{A_3}$ and $\mathbb P^4_{A_4}$ are given in Example \ref{example maximal linear spaces in F4 type}.
Then $(S,S_0)$ is not homologically rigid.
\end{proposition}

\begin{proof}

The space of  $\mathbb P^3$'s in $\mathbb S_{D_4}=\mathbb S_{B_3}$ has two connected components, $\mathcal H_1$ and $\mathcal H_2$, the former  containing $\mathbb P^3_{B_2}$ and the latter containing $\mathbb P^3_{A_3}$.
Each $\mathcal H_i$  is an orbit of $D_4$ and
is the union of two $B_3$-orbits, one is closed and the other is open  (Theorem 4.9 and Remark 5.10 of \cite{LM}).
 The closed $B_3$-orbit in $\mathcal H_1$ consists of translates of $\mathbb P^3_{B_2}$ by the group $B_3$. Thus $\mathbb P^3$'s in the open $B_3$-orbit in  $\mathcal H_1$ have the same homology class as $\mathbb P^3_{B_2}$ but are not translates of $\mathbb P^3_{B_2}$ by the group $B_3$.

 Consider $\mathbb S_{B_3}$ as a homogeneous submanifold of the rational homogeneous manifold $S$ of type $(F_4, \alpha_3)$ as in Example \ref{example maximal linear spaces in F4 type}. Then we have $\mathbb P^3_{B_2} \subset \mathbb S_{B_3} \subset S$ and we may regard $\mathcal H_1$ as the space of $\mathbb P^3$'s in $S$ contained in $\mathbb S_{B_3}$.
We claim that
$\mathbb P^3$'s in the open $B_3$-orbit in $\mathcal H_1$ are not translates of $\mathbb P^3_{B_2}$ by the group $F_4$.

For the proof of the claim observe first of all that the holomorphic tangent bundle $TS$ is equipped with a nontrivial $F_4$-invariant filtration corresponding to the gradation of the Lie algebra of $F_4$ defined by the marked Dynkin diagram $(F_4,\alpha_3)$.
Let $0 \neq D \subset TS$ be the minimal proper distribution in the filtration.  Taking intersection of distributions in the filtration with the holomorphic tangent bundle $T{\mathbb S_{B_3}}$ we obtain a 2-step filtration corresponding to the marked Dynkin diagram $(B_3,\alpha_3)$.  Let $0 \neq D' \subset T{\mathbb S_{B_3}}$ be the minimal proper distribution in the filtration.  We have $D' = D|_{\mathbb S_{B_3}} \cap T{\mathbb S_{B_3}}$.

 Let $\mathbb P$ be an element of $\mathcal H_1$.  Then, $\mathbb P$ is in the closed $B_3$-orbit of $\mathcal H_1$ if and only if $\mathbb P$ is tangent to $D'$ at every point $x \in \mathbb P$.   Let $\mathbb Q \in \mathcal H_1$ be an element of the open $B_3$-orbit in $\mathcal H_1$.  Then, $\mathbb Q$ is not tangent to $D'$, hence it is not tangent to $D$.  However, since $\mathbb P^3_{B_2}$ is tangent to $D'$ and hence {\it a fortiori} to $D$ everywhere on $\mathbb P^3_{B_2} \subset S$, and since $\Aut(S) = F_4$ preserves $D$, for any $\varphi \in F_4$, $\varphi(\mathbb P^3_{B_2})$ must be everywhere tangent to $D$, and it follows that $\mathbb Q \neq \varphi(\mathbb P^3_{B_2})$ for any $\varphi \in F_4$, proving the claim.  Hence, $\mathbb P^3_{B_2} $ is not homologically rigid in $S$.
  By the same arguments $\mathbb P^3_{A_3}$ is not   homologically rigid in $S$.

  Similarly, $\mathbb P^3_{A_4}$ is not homologically rigid in the rational homogeneous manifold of type $(F_4, \alpha_4)$.
\end{proof}

\begin{proposition} \label{maximal linear space case} Let $S=G/P$  be a rational homogeneous manifold associated to a simple root and let
$S_0$ be a maximal linear space in $S$. Then $(S,S_0)$ is Schur rigid except when $S$ is associated to a long root and  $S_0$ is a homogeneous submanifold of $S$ associated to a subdiagram $\mathcal D(S_0)$ of the marked Dynkin diagram $\mathcal D(S)$ of the following type:
\begin{enumerate}

\item $S=(B_n, \alpha_k)$ and $S_0=\mathbb P^{n-k}$ with $\Lambda =\{\alpha_{n}\}$ for $k=1$ and with $\Lambda =\{\alpha_{k-1}, \alpha_n\}$ for $2\leq k \leq n-2$;
\item $S=(C_n, \alpha_n)$ and $S_0=\mathbb P^1$ with $\Lambda =\{\alpha_{n-1}\}$;
\item $S=(F_4, \alpha_1)$ and $S_0=\mathbb P^2$ with $\Lambda =\{\alpha_3\}$;
\item $S=(G_2, \alpha_2)$ and $S_0 = \mathbb P^1$ with  $\Lambda =\{\alpha_1\}$,

\end{enumerate}
where $\Lambda$ denotes the set of simple roots in $\mathcal D(S) -\mathcal D(S_0)$ which are adjacent to $\mathcal D(S_0)$, or, when $S$ is associated to a short root and
$S_0$ is of the following form:
\begin{enumerate}
\item [(5)] $S=(F_4, \alpha_3)$ and $S_0 =\mathbb P^3_{B_2}$;
\item [(6)] $S=(F_4, \alpha_3)$ and $S_0 =\mathbb P^3_{A_3}$;
\item [(7)] $S=(F_4, \alpha_4)$ and $S_0 =\mathbb P^4_{A_4}$,
\end{enumerate}
where $\mathbb P^3_{B_2}$ and $\mathbb P^3_{A_3}$ and $\mathbb P^4_{A_4}$ are given in Example  \ref{example maximal linear spaces in F4 type}.
\end{proposition}

\begin{proof} Let $S=G/P$  be a rational homogeneous manifold associated to a simple root and let
$S_0$ be a maximal linear space in $S$.
As in the proof of Theorem \ref{main theorem I}, we  will apply Theorem \ref{sufficient condition for Schur rigidity} to get the Schur rigidity.
  The condition (I)  in Section \ref{sect: local characterizations}  is satisfied  because  $\mathcal C_x(S_0)$ is a linear space.  It suffices to show that  if $(S,S_0)$ is not in the list (1)--(7), then
\begin{enumerate}
\item [$\bullet$] $\mathcal C_x(S_0)$ satisfies the condition    ${\rm (II)}$ in Section \ref{sect: local characterizations} at a general point $x \in S_0$  and
\item [$\bullet$] $S_0$ intersects $gS_0$ in codimension 2 for any $g \in G$. \\
\end{enumerate}

Assume that $S$ is associated to a long simple root. Then a linear Schubert variety $S_0$ of $S$ is a homogeneous submanifold associated to a subdiagram $\mathcal D_0$ of $\mathcal D(S)$, and  $\mathcal C_x(S_0)$ satisfies (II)     (Lemma 4.1 of \cite{HoM}).

 If $(S, S_0)$ is not of the form (1) - (4), then,     for any $\gamma \in \Lambda$, there are at least two roots   $\alpha$ with $\frak g_{\alpha} \subset T_{x_0}S_0$ such that $\langle \alpha, \gamma \rangle \not=0$. Therefore, as in the proof of Proposition \ref{I II codimension}, $T_xS_0 \cap T_x(s_{\gamma}S_0)$ has codimension $\geq 2$ in $T_xS_0$ and hence
 $S_0 \cap gS_0$ has codimension $\geq 2$ for any $g \in G$ with $S_0 \not=gS_0$.
 By Theorem \ref{sufficient condition for Schur rigidity}, $(S,S_0)$ is Schur rigid.

If $(S,S_0)$ is one of the forms in (1) - (3), it is not Schubert rigid   by Theorem \ref{thm Hong park}  and thus is not Schur rigid.  If $(S,S_0)$ is  of the form (4), then $S_0$ is a minimal rational curve in $S$. 
 Schur rigidity must fail for $(S,S_0)$ whenever $S$ is
not linear and $S_0$ is a minimal rational curve (and a Schubert cycle) for the following elementary observation.
Embed $S$ into a projective space $\mathbb P^N$ by the minimal canonical embedding (i.e., by $\mathcal O(1)$) and suppose the homology
class of $S$ in $\mathbb P^N$ is $k$ times the positive generator of the $H_{2n}(S,\mathbb Z)$, where $n = \dim(S)$.  Then, a general linear section of dimension 1 is a reduced and irreducible smooth curve of degree  $k > 1$.
Hence the pair $(S, S_0)$ is not Schur rigid. This completes the proof for the case where $S$ is associated to a long root.  \\

Assume that  $S$ is associated to a short simple root. Then
$\mathcal C_x(S)$ and its maximal linear spaces are given as follows  (\cite{LM}, Lemma 5.1 of \cite{HoPa}): \\

\begin{enumerate}
\item [1.] $(C_{\ell}, \alpha_k), k \geq 2$:
  $\mathcal C_x(S) $ is $\mathbb P( \{ u \otimes q + c u^2 : u \in U, q \in Q, c \in \mathbb C \}  ) \subset \mathbb P((U \otimes Q) \oplus S^2U)$ and its maximal linear space is

\vskip 5 pt
\begin{enumerate}
  \item[1.a] $\mathbb P(\{u \otimes q + cu^2: q \in Q, c \in \mathbb C \}) \simeq \mathbb P^{2m}$ for some $u \in U$ or
  \item[1.b] $\mathbb P(\{u  \otimes   q:u \in U\}) \simeq \mathbb P^{k-1} $ for some $ q  \in  Q $
  \end{enumerate}

\vskip 5 pt \noindent
 where $U$ is a vector space of dimension $k$ and $Q$ is a vector space of dimension $2m:=2 \ell-2k$. \\

\item  [2.] $(F_4, \alpha_3)$:
$\mathcal C_x(S)$ is $\mathbb P(\{e^* \otimes q + (e_1^* \wedge e_2^*) \otimes q^2 :e \wedge e_1 \wedge e_2=0 , e, e_1, e_2 \in E, q \in Q\})$ and its maximal linear space is

\vskip 5 pt
\begin{enumerate}
\item[2.a] $\mathbb P^2$'s in $\mathbb P(\{e^* \otimes q + (e_1^* \wedge e_2^*) \otimes q^2: e\wedge e_1\wedge e_2=0, e,e_1,e_2 \in E\}) \simeq \mathbb Q^4$ for some $q \in Q$ or

  \item[2.b] $ \mathbb P(\{e^* \otimes  q: q \in Q\}) \simeq \mathbb P^1$ for some $e^* \in E^*$
 \end{enumerate}

\vskip 5 pt \noindent
where $E$ is a vector space of dimension $3$ and $Q$ is a vector space of dimension $2$. \\

\item [3.]  $(F_4, \alpha_4)$:
$\mathcal C_x(S) $ is the closure of $L$-orbit of $[v_1 + v_2]$ in $\mathbb P(V_1 \oplus V_2)$ and its maximal linear space is

\vskip 5 pt
\begin{enumerate}
   \item[3.a] $\mathbb P^4=$ the cone over $\mathbb P^3_{B_2}$ in $\mathbb S_{B_3} \subset \mathbb P(V_1)$ with the vertex at $[v_2]$ and its $L$-translates, or
   \item[3.b] $\mathbb P^3=$ the cone over $\mathbb P^2$ in $\mathbb Q^5 \subset \mathbb P(V_2)$ with the vertex at $[v_1]$ and its $L$-translates, or
   \item[3.c] $\mathbb P^3_{A_3}$'s  in $\mathbb S_{B_3}   \subset \mathbb P(V_1)$,
 \end{enumerate}

\vskip 5 pt

 \noindent
   where $L$, the simple group of type $B_3$,  is the semisimple part of the isotropy group of $G$ at $x$, and $V_1$ is the spin representation of $L$ and $V_2$ is the standard representation of $L$, and $v_1$ is a highest weight vector of $V_1$ and $v_2$ is a highest weight vector of $V_2$, and $\mathbb S_{B_3}  $ is the highest weight orbit in $\mathbb P(V_1)$ and $\mathbb Q^5$ is the highest weight orbit in $\mathbb P(V_2)$. \\
\end{enumerate}

If $\mathcal C_x(S_0)$ is  of the form  1.b or 2.b, then  $S_0$ does not contain a general line and  Theorem \ref{sufficient condition for Schur rigidity} cannot apply directly.
However, there is an embedding of $S$ into another  rational homogeneous manifold $S'$ in such a way that $S_0$ is a Schubert variety of $S'$, and the homological rigidity of $(S,S_0)$ follows from the homological rigidity of $(S', S_0)$ (See the proof of Proposition 3.6 of \cite{HoM}).   Since  $(S', S_0)$ is Schur rigid, so is $(S,S_0)$.

  If $\mathcal C_x(S_0)$ is of the form 1.a, then any local deformation of $\mathcal C_x(S_0)$ is of the same form $\mathbb P(\{u' \otimes q + c{u'}^2: q \in Q, c \in \mathbb C \}) \simeq \mathbb P^{2m}$ but for  a different  $u' \in U$. Thus $\mathcal C_x(S_0)$ satisfies the condition (II).  Since $\mathcal C_x(S_0) \cap \mathcal C_x(gS_0)$ is empty, $S_0 \cap gS_0$ has dimension $\leq 0$ because $S_0$ is linear. From  $ \dim S_0  \geq 2$ it follows that $S_0 \cap gS_0$ has codimension $\geq 2$ for any $g \in G$.

If $\mathcal C_x(S_0)$ is of the form 3.a or 3.b, any deformation of $\mathcal C_x(S_0)$ in $\mathcal C_x(S)$ is again a maximal linear space in $\mathcal C_x(S_0)$ and thus is an $L$-translate of $\mathcal C_x(S_0)$. Hence $\mathcal C_x(S_0)$ satisfies the condition (II). The space of $\mathbb P^n$'s in $\mathbb Q^{2n}$ has two connected components, and any two $\mathbb P^n$ in the same connected component intersects in codimension two. Thus,
the intersection $\mathcal C_x(S_0)\cap\mathcal  C_x(gS_0)$ has  codimension $\geq 3$ in $\mathcal C_x(S_0)$. It follows that $S_0 \cap gS_0$ has  codimension $\geq 3$ in $S_0$ because $S_0$ is linear.

Therefore, if $\mathcal C_x(S_0)$ is neither of the form 2.a nor of the form 3.c, then $(S,S_0)$ is Schur rigid by Theorem \ref{sufficient condition for Schur rigidity}. \\

If  $\mathcal C_x(S_0)$ is of the form 2.a, then $S_0$ is contained in the homogeneous submanifold $S_1$ of $S$ associated to the subdiagram of type $(B_3, \alpha_3)$ obtained by $\Lambda=\{\alpha_4\}$. Thus, $S_0$ is either $\mathbb P^3_{B_2}$   or  $\mathbb P_{A_3}^3$ embedded in $S_1\simeq \mathbb S_{B_3}$.
If  $\mathcal C_x(S_0)$ is of the form 3.c, then $S_0$ is $\mathbb P^4_{A_4}$.
Therefore,  if $\mathcal C_x(S_0)$ is of the form 2.a or 3.c, i.e., $(S,S_0)$ is one of the forms (5) - (7), then it is   not homologically rigid by Proposition \ref{F4 not homologically rigid} and thus it  is not Schur rigid.
\end{proof}


\begin {thebibliography} {XXX}





\bibitem {BB} A. Bialynicki-Birular, \emph{Some theorem on actions of algebraic groups}, Ann. Math.  98  (1973) 480--497.

\bibitem {BB76} A. Bialynicki-Birular, \emph{Some properties of the decompositions of algebraic varieties determined by actions of a torus}, Bull. Acad. Polon. Sci. S\'er. Sci. Math. Astronom. Phys.  24  (1976)  no. 9, 667--674.




 \bibitem{BrPo99} M. Brion and P. Polo, \emph{Generic singularities of certain Schubert varieties}, Math. Z.   231  (1999) 301--324


 \bibitem {Br} M. Brion, \emph{Lectures on the geometry of flag varieties}, Topics in cohomological studies of algebraic varieties,  Trends Math., Birkh\"auser, Basel, 2005, pp.\,33--85.

\bibitem {B} R. Bryant, \emph{Rigidity and quasi-rigidity of extremal
cycles in compact Hermitian symmetric spaces}, math. DG/0006186.




\bibitem{CoRo13} I. Coskun and C. Robles, \emph{Flexibility of Schubert classes},  Differential Geom. Appl.  31  (2013),  no. 6, 759--774.

\bibitem{Co11} I. Coskun, \emph{Rigid and non-smoothable Schubert classes}, J. Differential Geom.  87  (2011),  no. 3, 493--514.



 \bibitem{Co14} I. Coskun, \emph{Rigidity of Schubert classes in orthogonal Grassmannians}, Israel J. Math.  200  (2014),  no. 1, 85--126.

  \bibitem{Co18} I. Coskun, \emph{Restriction varieties and the rigidity problem},  Schubert varieties, equivariant cohomology and characteristic classes IMPANGA 15, 49--95, EMS Ser. Congr. Rep., Eur. Math. Soc., Z\"urich, 2018.






\bibitem{Deo} V. V. Deodhar, \emph{On some geometric aspects of Bruhat orderings. I. A finer decomposition of Bruhat cells} Invent. Math. 79 (1985)  499--511.


\bibitem{Fu} W. Fulton, Intersection Theory,  Ergebnisse der Mathematik und ihrer Grenzgebiete. 3. Folge. A Series of Modern Surveys in Mathematics [Results in Mathematics and Related Areas. 3rd Series. A Series of Modern Surveys in Mathematics], 2. Springer-Verlag, Berlin, 1998.


\bibitem{HRT} R. Hartshorne, E. Rees, and E. Thomas, \emph{Nonsmoothing of algebraic cycles on Grassmann varieties}, Bull. Amer. Math. Soc. 80 (1974)  847--851.


\bibitem{HoK} J. Hong and M. Kwon, {\it Rigidity of smooth Schubert varieties in a rational homogeneous manifold associated to a short root}, arXiv:1907.09694.

\bibitem {HoM09} J. Hong and N. Mok, \emph{Analytic continuation of holomorphic maps respecting varieties of minimal rational tangents and applications to rational homogeneous manifolds},  J. of Differential Geom.   86  (2010) no.3, 539--567.

\bibitem {HoM} J. Hong and N. Mok, \emph{Characterization of smooth Schubert varieties in rational homogeneous manifolds   of Picard number 1}, J. Algebraic Geom. 22 (2013) no. 2, 333--362.


\bibitem{HoPa} J. Hong and K.-D. Park, \emph{Characterization of Standard Embeddings between Rational
Homogeneous Manifolds of Picard Number 1}, Int. Math. Res. Notices,   2011  (2011) 2351--2373.



\bibitem{Ho05} J. Hong, \emph{Rigidity of singular Schubert varieties in $Gr(m, n)$}, J. Differential Geom.  71  (2005)  no. 1, 1--22.


\bibitem {Ho07} J. Hong, \emph{Rigidity of smooth Schubert varieties in Hermitian symmetric spaces}, Trans. Amer. Math. Soc. 359 (2007) 2361--2381.


\bibitem  {HwM98} J.-M. Hwang   and N. Mok,  \emph{Rigidity of irreducible Hermitian symmetric spaces of the compact type under Kahler deformation}, Invent. Math.  131 (1998)  no. 2, 393--418.



\bibitem  {HwM99} J.-M. Hwang and N. Mok, \emph{Varieties of minimal rational tangents on uniruled projective manifolds}, in:M. Schneider, Y.-T. Siu (Eds.), Several complex variables, MSRI publications, Vol 37, Cambridge University Press, 1999, pp. 351--389.



 \bibitem {HwM01} J.-M.  Hwang and N. Mok, \emph{Cartan-Fubini type extension of holomorphic maps for Fano manifolds of Picard number 1}, J. Math. Pure Appl.   80 (2001) 563--575.

 \bibitem {HwM02} J.-M.  Hwang and N. Mok, \emph{Deformation rigidity of the rational homogeneous space associated to a long simple root}, Ann. Scient. \'Ec. Norm. Sup., $4^e$ s\'erie, t. 35. (2002)  173--184.

 \bibitem {HwM04b} J.-M.  Hwang and N. Mok,  \emph{Deformation rigidity of the 20-dimensional $F\sb 4$-homogeneous space associated to a short root}. Algebraic transformation groups and algebraic varieties, 37--58, Encyclopedia Math. Sci., 132, Springer, Berlin, 2004.

\bibitem  {HwM05} J.-M. Hwang and N. Mok, \emph{Prolongations of infinitsimal linear automorphisms of projective varieties and rigidity of rational homogeneous spaces of Picard number 1 under K\"ahler deformation}, Invent. Math. 160 (2005)  no. 3, 591--645.




\bibitem{Ke}G. R.  Kempf, \emph{On the collapsing of homogeneous bundles}, Invent. Math. 37 (1976) 229--239.


\bibitem {K}  S. Kleiman, \emph{The transversality of a general translate}, Compos. Math. 28 (1974)  287--297.


\bibitem {Kd63} K. Kodaira, \emph{On stability of compact submanifolds of complex manifolds}, Amer. J. Math.   85 (1963) 79--94.



\bibitem{Ko63} B.  Kostant, \emph{Lie algebra cohomology and generalized Schubert cells}. Ann. Math  77  (1963) no.1, 72--144.

\bibitem {LM} J.M Landsberg and L. Manivel, \emph{On the projective geometry of rational homogeneous varieties},  Comment. Math. Helv. 78 (2003)  no. 1, 65--100.

\bibitem{Ma}  O. Mathieu : Formules de caract\`eres pour les alg\`ebres de Kac-Moody g\'en\'erales,
Ast\'erisque   159-160  (1988).




\bibitem{MkZh} N. Mok and Y.  Zhang, \emph{Rigidity of pairs of rational homogeneous spaces of Picard number 1 and analytic continuation of geometric structures on uniruled manifolds},   J. Differential Geom. 112 (2019) 263--345.


\bibitem{Mk16} N. Mok, \emph{Geometric structures and substructures on uniruled projective manifolds}, in:  P. Cascini, J. McKernan and J.V.Pereira (Eds.) Foliation Theory in Algebraic Geometry (Simons Symposia), Springer-Verlarg,  2016, pp.103--148.

\bibitem{Pa} Pasquier, B.:  On some smooth projective two-orbit  varieties with Picard number one. Math. Ann.  344 (2009) no. 4,   963--987.


\bibitem{RoT12} C. Robles and D. The, \emph{Rigid Schubert varieties in compact Hermitian symmetric spaces},  Selecta Math. (N.S.)  18  (2012)  no. 3, 717--777.

 \bibitem{Ro13} C.  Robles, \emph{Schur flexibility of cominuscule Schubert varieties}, Comm. Anal. Geom.  21  (2013) no. 5, 979--1013.



\bibitem{Sp} T. A. Springer, Linear algebraic groups, second edition, Progress in Math. 9, Birkh\"auser Boston, 1998.


 \bibitem{W} M. Walters, \emph{Geometry and uniqueness of some extremal subvarieties in complex Grassmannians}, Ph.D thesis, U. Michigan, 1997.

\end {thebibliography}
\end{document}